\documentclass[11pt, reqno]{amsart}    

\usepackage{amsmath,amssymb,latexsym}
\usepackage{color}
\usepackage{xcolor}
\usepackage{graphicx}
\usepackage{cite}
\usepackage[colorlinks=true]{hyperref}
\hypersetup{urlcolor=blue, citecolor=red, linkcolor=blue}

\newcommand{\R}{\mathbb{R}}

\newcommand{\Z}{\mathbb{Z}}

\newcommand{\bigO}{\mathcal{O}}
\newcommand{\diff}{\mathrm{d}}

\newtheorem{theorem}{Theorem}[section]
\newtheorem{lemma}[theorem]{Lemma}
\newtheorem{proposition}[theorem]{Proposition}

\newtheorem*{main-theorem}{Main Theorem}
\newtheorem*{remark*}{Remark}
\newtheorem*{lemma*}{Lemma A.1}

\numberwithin{equation}{section}

\begin{document}

\title[Long time behavior of  the fractional Korteweg-de Vries]{ Long time behavior of  the fractional Korteweg-de Vries equation with cubic nonlinearity}

\author{Jean-Claude Saut}

\author{Yuexun Wang}

\address{ Universit\' e Paris-Saclay, CNRS, Laboratoire de Math\'  ematiques d'Orsay, 91405 Orsay, France.}
\email{jean-claude.saut@universite-paris-saclay.fr}

\address{	
	School of Mathematics and Statistics, 
	Lanzhou University, 370000 Lanzhou, China.}
\address{Universit\' e Paris-Saclay, CNRS, Laboratoire de Math\'  ematiques d'Orsay, 91405 Orsay, France.}
\email{yuexun.wang@universite-paris-saclay.fr}

\thanks{}


\begin{abstract} We prove global existence and modified scattering for the solutions of the Cauchy problem to the fractional Korteweg-de Vries equation with cubic nonlinearity for small, smooth and localized initial data. 
\end{abstract}
\maketitle

\section{Introduction}
We consider the fractional Korteweg-de Vries equation with cubic nonlinearity:
\begin{align}\label{eq:main}
	\partial_t u-|D|^{\alpha} \partial_x u=-u^2\partial_x u, \quad -1<\alpha<0,
\end{align}
where \(u\) maps \(\R_t\times\R_x\) to \(\R\) and \(|D|^{\alpha}\) is the usual Fourier multiplier operator with the symbol \(|\xi|^\alpha\).  The initial data is given by
\begin{align}\label{eq:initial}
	u(0,x)=u_0(x).
\end{align}

Although it does not seem to appear in a physical context, this equation, which will be from now on referred to as the modified fractional KdV equation (modified fKdV),  is a good toy model to understand the influence of a weak dispersion on the dynamics of a scalar conservation law such as the modified Burgers equation. When $\alpha=-\frac{1}{2},$ it is reminiscent for large frequencies, of a modified Whitham equation 

\begin{equation}\label {modWhit}
	\partial_t u+\mathcal L \partial_xu=-u^2\partial_xu,
\end{equation}
where the Fourier multiplier operator $\mathcal L$ with symbol $\left(\frac{\tanh \xi}{\xi}\right)^{1/2}.$  

We refer for instance to \cite{KLPS} for various issues and results on the usual (quadratic) Whitham equation, in particular on its KdV, long wave limit. It is straightforward to check that Theorem 1 in \cite{KLPS} extends to the modified Whitham equation, proving its long wave limit to the modified KdV equation.
More precisely we consider the rescaled modified Whitham equation
\begin{equation}\label{mWhit}
	u_t+\mathcal L_\epsilon u_x+\epsilon u^2u_x=0,
\end{equation}
where the non local operator $\mathcal L_\epsilon$ is related to the dispersion relation of the (linearized) water waves system and is defined by
$$\mathcal L_\epsilon =l(\sqrt{\epsilon}D):=\left(\frac{\tanh \sqrt \epsilon |D|}{\sqrt \epsilon |D|}\right)^{1/2} \quad \text{and} \quad D=-i\nabla=-i\frac{\partial}{\partial x}, $$
that we want to compare to the modified KdV equation

\begin{equation}\label{mKdV}
	v_t+v_x+\epsilon v^2v_x+\epsilon v_{xxx}=0.
\end{equation}
It is obvious that for any initial data $\phi\in H^s(\R), s>\frac{3}{2}$ \eqref {mWhit} admits a unique solution u in 
$C([0,T_\epsilon);H^s(\R))$ where $T_\epsilon = O(1/\epsilon).$
Denoting $v$ the solution of\eqref{mKdV} with the same initial data $\phi$, one obtains, proceeding as in the proof of Theorem 1 in \cite{KLPS} which considered the Whitham equation \eqref{Whit}:

\begin{theorem}
	Let $\phi \in H^{\infty}(\mathbb R)$. Then, for all $j \in \mathbb N$, $j \ge 0$, there exists
	$M_j=M_j(\|\phi\|_{H^{j+8}})>0$  such that
	\begin{equation} \label{maintheo.1}
		\|(u-v)(t)\|_{H^j_x} \le M_j \epsilon^2t,
	\end{equation}
	for all $0 \le t \lesssim \epsilon^{-1}$.
\end{theorem}

It is well known that the modified Burgers equation undergoes shock formation, even for arbitrary small  smooth  initial data, provided  the square of the initial has a negative slope at some point. In particular no global solutions exist for arbitrary small smooth initial data in Sobolev spaces, such as a small gaussian.

The question we address here is whether this property is still true when adding a weakly dispersive term as in \eqref{eq:main}.  In fact we prove that adding this dispersive term allows the existence and (modified) scattering of small solutions.

Throughout the paper, we will always use the notation \(f(t)=e^{-t|D|^{\alpha} \partial_x}u(t)\) to denote the profile of \(u\). By time reversibility we  need only to consider the existence for positive time.

\vspace{0.3cm}
Our main result can be stated precisely as follows:
\begin{theorem}\label{th:main} Let \(\alpha\in(-1,0)\) and define the \(Z\)-norm 
\[\|g\|_Z=\|(1+|\xi|)^{10}\widehat{g}(\xi)\|_{L^\infty_\xi}.\] 
 Assume that \(N_0=100,\ p_0\in (0,1/1000]\cap (0,-\frac{\alpha}{100}] \) are fixed, and \(u_0\in H^{N_0}(\R)\) satisfies
\begin{align}\label{1}
\|u_0\|_{H^{N_0}}+\|u_0\|_{H^{1,1}}+\|u_0\|_Z=\varepsilon_0\leq \bar{\varepsilon},
\end{align}
for some constant \(\bar{\varepsilon}\) sufficiently small (depending only on \(\alpha\) and \(p_0\)). Then the Cauchy problem \eqref{eq:main}-\eqref{eq:initial} admits a unique
 global solution \(u\in C((\R):H^{N_0}(\R))\) satisfying the following uniform bounds for \(t\geq 1\)
\begin{align}\label{2}
t^{-p_0}\|u\|_{H^{N_0}}+t^{-p_0}\|f\|_{H^{1,1}}+\|f\|_Z\lesssim \varepsilon_0.
\end{align}
Moreover,  there exists \(w_\infty\in L^\infty(\R)\) such that  for \(t\geq 1\)
\begin{align}\label{3}
t^{p_0}\left\|\exp\left(\frac{3\mathrm{i}\xi|\xi|^{1-\alpha}}{\alpha(\alpha+1)}\int_1^t|\widehat{f}(s,\xi)|^2\frac{\diff s}{s}\right)(1+|\xi|)^{10}\widehat{f}(\xi)-w_\infty(\xi)\right\|_{L^\infty_\xi}\lesssim \varepsilon_0.
\end{align}
\end{theorem}

The local well-posedness on the time interval \([0,1]\) for \eqref{eq:main}-\eqref{eq:initial} is standard provided $||u_0||_{H^2}$ is small enough, in particular  under the smallness assumption \eqref{1}. Then the existence and uniqueness of global solutions may be constructed by a bootstrap argument which allows  to extend the local solutions. More precisely,
assume that the following \(X\)-norm is a priori small:
\begin{equation}\label{4}
	\begin{aligned}
		\|u\|_{X}&=\sup_{t\geq 1}\bigg(t^{-p_0}\|u\|_{H^{N_0}}+t^{-p_0}\|f\|_{H^{1,1}}+\|f\|_Z\bigg)\leq \varepsilon_1
	\end{aligned}
\end{equation}
with \(\varepsilon_1=\varepsilon_0^{1/3}\), 
we then aim to show that the above a priori assumption may be improved to
\begin{align}\label{5}
	\|u\|_{X}\leq C(\varepsilon_0+\varepsilon_1^3),
\end{align}
for some absolute constant \(C>1\). 

The litterature dealing with the problem of global existence and scattering of small solutions to nonlinear dispersive equations is vast and we refer to the Introduction of \cite{MR3121725} for a useful survey.

The present work is close to \cite{MR3961297,MR3552008,MR3121725} in methodology.
In comparison with \cite{MR3121725}, the presence of a derivative
on the nonlinearity in the equation \eqref{eq:main} plays a crucial role in our proof in the sense that on the one hand it eliminates part of resonances in low frequencies and on the other hand it avoids using \(\partial_s\widehat{f}_l\) in \(L^\infty\)-norm in some frequency regimes, which allows us to extend the estimates to all \(\alpha\in(-1,0)\).

To end this section, we list the notations frequently used throughout the paper.

We denote by \(\mathcal{F}(g)\) or \(\widehat{g}\) the Fourier transform of a Schwartz function \(g\) whose formula is given by 
\begin{equation*}
	\begin{aligned}
		\mathcal{F}(g)(\xi)=\widehat{g}(\xi):=\frac{1}{\sqrt{2\pi}}\int_{\R}g(x)e^{-\mathrm{i}x\xi}\,\diff x
	\end{aligned}
\end{equation*}
with inverse
\begin{equation*}
	\begin{aligned}
		\mathcal{F}^{-1}(g)(x)=\frac{1}{\sqrt{2\pi}}\int_{\R}g(\xi)e^{\mathrm{i}x\xi}\,\diff \xi,
	\end{aligned}
\end{equation*}
and by \(m(\partial_x)\) the Fourier multiplier with symbol \(m\) via the relation 
\begin{equation*}
	\begin{aligned}
		\mathcal{F}\big(m(\partial_x)g\big)(\xi)=m(\mathrm{i}\xi)\widehat{g}(\xi).
	\end{aligned}
\end{equation*}

Let \(\varphi\) be a smooth function satisfying
\begin{equation*}
	\varphi(\xi)=
	\left\{ 
	\begin{aligned}
		1,\quad &|\xi|\leq 1\\
		0,\quad &|\xi|>2.
	\end{aligned}
	\right.
\end{equation*}
Set 
\begin{equation*}
	\begin{aligned}
		\psi(\xi)=\varphi(\xi)-\varphi(2\xi),\quad \psi_j(\xi)=\psi(2^{-j}\xi),\quad \varphi_j(\xi)=\varphi(2^{-j}\xi),
	\end{aligned}
\end{equation*}
we then may define the  
Littlewood-Paley projections \(P_j,P_{\leq j},P_{> j}\)  via 
\begin{equation*}
	\begin{aligned}
		\widehat{P_jg}(\xi)=\psi_j(\xi)\widehat{g}(\xi),\quad \widehat{P_{\leq j}g}(\xi)=\varphi_j(\xi)\widehat{g}(\xi),\quad P_{> j}=1-P_{\leq j},
	\end{aligned}
\end{equation*}
and also \(P_{\sim},P_{\lesssim j},P_{\ll j}\) by 
\begin{equation*}
	\begin{aligned}
		P_{\sim j}=\sum_{2^k\sim 2^j}P_k, \quad P_{\lesssim j}=\sum_{2^k\leq 2^{j+C}}P_k,\quad P_{\ll j}=\sum_{2^k\ll 2^j}P_k,
	\end{aligned}
\end{equation*}
and the obvious notation for \(P_{[a,b]}\).
We will also denote \(g_j=P_jg, g_{\lesssim j}=P_{\lesssim j}g\), and so on, for convenience.

The notation \(C\)  always denotes a nonnegative universal constant which may be different from line to line but is
independent of the parameters involved. Otherwise, we will specify it by  the notation \(C(a,b,\dots)\).
We write \(g\lesssim h\) (\(g\gtrsim h\)) when \(g\leq  Ch\) (\(g\geq  Ch\)), and \(g \sim h\) when \(g \lesssim h \lesssim g\).
We also write \(\sqrt{1+x^2}=\langle x\rangle\) for \(x\in\R\) and \(\|g\|_{H^{1,1}}=\|\langle x\rangle g\|_{H^1}\), and \(P_{[k-2,k+2]}:=P_k^\prime\) and for simplicity.

\section{Decay Estimates}
This section is devoted to presenting some decay estimates of the solution of the equation \eqref{eq:main}. We first recall a dispersive linear estimate from \cite[Lemma 2.3]{MR3961297}:
\begin{lemma} For any \(t\geq 1\),  the following linear dispersive estimates hold:
	\begin{align}\label{6}
	\left\|e^{t|D|^{\alpha}\partial_x}P_kg\right\|_{L^\infty}
\lesssim t^{-\frac{1}{2}}2^{\frac{1-\alpha}{2}k}\|\widehat{g}\|_{L^\infty}
+t^{-\frac{3}{4}}2^{-\frac{1+3\alpha}{4}k}(\|\widehat{g}\|_{L^2}+2^k\|\partial\widehat{g}\|_{L^2}),
	\end{align}
and
	\begin{align}\label{6.5}
	\left\|e^{t|D|^{\alpha}\partial_x}P_kg\right\|_{L^\infty}
	\lesssim t^{-\frac{1}{2}}2^{\frac{1-\alpha}{2}k}\|g\|_{L^1},
	\end{align}
for \(\alpha\in (-1,1)\setminus\{0\}\).
\end{lemma}

We next show the following decay estimates for the solution: 
\begin{lemma} 
Let \(\alpha\in (-1,0)\) and \(u\) be the solution to \eqref{eq:main}. Assume that
\begin{align*}
	t^{-p_0}\|u\|_{H^{N_0}}+(1+t)^{-p_0}\|f\|_{H^{1,1}}+\|f\|_Z\leq 1,
	\end{align*}
for any \(t\geq 1\), then we have
	\begin{align}\label{7}
	\|u\|_{L^\infty}+\|\partial_xu\|_{L^\infty}\lesssim t^{-1/2},
	\end{align}
for any \(t\geq 1\).
\end{lemma}
\begin{proof}  {\bf{Case  \(\alpha\in (-1,-\frac{1}{3}]\).}}
Observe  \(\frac{1-\alpha}{2}\in[\frac{2}{3},1)\) and \(-\frac{1+3\alpha}{4}\in[0,\frac{1}{2})\).  To prove  that \(\|u\|_{L^\infty}\) satisfy \eqref{7}, we first show 
\begin{equation}\label{8}
	\begin{aligned}
	\quad \sup_{k\in\Z}(2^{2 k}\left\|P_ku\right\|_{L^\infty})
	\lesssim t^{-\frac{1}{2}}.
	\end{aligned}
	\end{equation}
 In the frequency regime \(2^k\geq t^{(1-4p_0)/10}\),  we use \eqref{6.5} and
 \eqref{75} to deduce that 
	\begin{equation}\label{8.5}
	\begin{aligned}
	&\quad 2^{2 k}\left\|P_ku\right\|_{L^\infty}\lesssim t^{-\frac{1}{2}}2^{2 k}2^{\frac{1-\alpha}{2}k}\|P_k^\prime f\|_{L^1}\\
	&\lesssim t^{-\frac{1}{2}}2^{2 k}2^{\frac{1-\alpha}{2}k}2^{-\frac{k}{2}}\|\widehat{P_k^\prime f}\|_{L^2}^{1/2}(\|\widehat{P_k^\prime f}\|_{L^2}+2^k\|\partial\widehat{P_k^\prime f}\|_{L^2})^{1/2}\\
	&\lesssim t^{-\frac{1}{2}}2^{2 k}2^{\frac{1-\alpha}{2}k}2^{-\frac{k}{2}}2^{-\frac{N_0k}{2}}\|P_k^\prime f\|_{H^{N_0}}^{1/2}(\|\widehat{P_k^\prime f}\|_{L^2}+2^k\|\partial\widehat{P_k^\prime f}\|_{L^2})^{1/2}\\
	&\lesssim t^{-\frac{1}{2}}t^{\frac{1-4p_0}{10}(2+\frac{1-\alpha}{2}-\frac{1}{2}-\frac{N_0}{2})} t^{p_0},
	\end{aligned}
	\end{equation}
which gives a stronger bound than what we need. 
We then consider the frequency regime \(2^k\leq t^{(1-4p_0)/10}\).   
It follows from \eqref{6} that
 	\begin{equation}\label{9}
	\begin{aligned}
	\left\|P_ku\right\|_{L^\infty}
\lesssim t^{-\frac{1}{2}}2^{\frac{1-\alpha}{2}k}\big\|\widehat{P_k^\prime f}\big\|_{L^\infty}
+t^{-\frac{3}{4}}2^{-\frac{1+3\alpha}{4}k}\big(\big\|\widehat{P_k^\prime f}\big\|_{L^2}+2^k\big\|\partial\widehat{P_k^\prime f}\big\|_{L^2}\big).
	\end{aligned}
	\end{equation}
Now the desired bound \eqref{8} is a consequence of  \eqref{9} and of  the following facts
	\begin{equation*}
	\begin{aligned}
	\frac{2^{(\frac{1-\alpha}{2}+2)k}}{1+2^{10k}}\lesssim 1,\quad 2^{(-\frac{1+3\alpha}{4}+2)k} \lesssim t^{\frac{1}{4}-p_0},
	\end{aligned}
	\end{equation*}
	and 
	\begin{equation*}
	\begin{aligned}
	(1+2^{10k})\big\|\widehat{P_k^\prime f}\big\|_{L^\infty}\lesssim 1,\quad \|\widehat{P_k^\prime f}\|_{L^2}+2^k\|\partial\widehat{P_k^\prime f}\|_{L^2}\lesssim t^{p_0}.
	\end{aligned}
	\end{equation*}

	We are ready to show \(\|u\|_{L^\infty}\) satisfy \eqref{7}. 
       It first follows from \eqref{8} that 
	 \begin{equation*}
	 \begin{aligned}
	  \sum_{0\leq k\in\Z}\|P_ku\|_{L^\infty}= \sum_{0\leq k\in\Z}2^{-2 k}(2^{2 k}\left\|P_k u\right\|_{L^\infty})
	 \lesssim t^{-\frac{1}{2}} \sum_{0\leq k\in\Z}2^{-2 k}\lesssim t^{-\frac{1}{2}} .
	 \end{aligned}
	 \end{equation*}
	Thus it  remains to estimate \( \sum_{0> k\in\Z}\|P_ku\|_{L^\infty}\).  When \(2^k\leq t^{-1}\), the desired bound follows from the following inequality after summation in \(k\) 
	 \begin{equation*}
	 \begin{aligned}
	 \|P_k u\|_{L^\infty}
	 \lesssim 2^{k/2}\|u\|_{L^2}.
	 \end{aligned}
	 \end{equation*}
where we have used Bernstein inequality.
For  \(1\geq 2^k\geq t^{-1}\),  it follows from \eqref{9} that
	 \begin{equation}
	 \begin{aligned}
	 \sum_{1\geq 2^k\geq t^{-1}}\|P_k u\|_{L^\infty}
	 &\lesssim t^{-\frac{1}{2}}\sum_{1\geq 2^k\geq t^{-1}}\frac{2^{\frac{1-\alpha}{2}k}}{1+2^{10k}}
	 \big\|(1+2^{10k})\widehat{P_k^\prime f}\big\|_{L^\infty}\\
	 &\quad+t^{-\frac{3}{4}}t^{p_0}\sum_{1\geq 2^k\geq t^{-1}}2^{-\frac{1+3\alpha}{4}k}\\
	 &\lesssim t^{-\frac{1}{2}}
	 +t^{-\frac{3}{4}}t^{p_0}\log t \lesssim t^{-\frac{1}{2}}.
	 \end{aligned}
	 \end{equation}

We finally prove the bound for \(\|\partial_xu\|_{L^\infty}\).  One has
	\begin{equation*}
	\begin{aligned}
	&\quad\|\partial_xu\|_{L^\infty}\leq \sum_{k\in\Z}\|P_k\partial_xu\|_{L^\infty}\lesssim\sum_{k\in\Z}2^k\|P_ku\|_{L^\infty}\\
	&=\sum_{0\geq k\in\Z}2^k\|P_ku\|_{L^\infty}+\sum_{0\leq k\in \Z}2^{-k}(2^{2 k}\|P_ku\|_{L^\infty})\\
	&\lesssim \|u\|_{L^\infty}\sum_{0\geq k\in\Z}2^k+t^{-\frac{1}{2}}\sum_{0\leq k\in \Z}2^{-k}
\lesssim t^{-\frac{1}{2}}.
	\end{aligned}
	\end{equation*}

\noindent{\bf{Case  \(\alpha\in (-\frac{1}{3},0)\).}}
Observe that   \(\frac{1-\alpha}{2}\in(\frac{1}{2},\frac{2}{3})\) and \(-\frac{1+3\alpha}{4}\in(-\frac{1}{4},0)\). We now split the frequencies into \(2^k\geq t^{(1-4p_0)/8}\)
and \(2^k\leq t^{(1-4p_0)/8}\). In a similar fashion as before, one still has 
\begin{equation*}
\begin{aligned}
\quad \sup_{k\in\Z}(2^{2 k}\left\|P_ku\right\|_{L^\infty})
\lesssim t^{-\frac{1}{2}},
\end{aligned}
\end{equation*}
and furthermore shows 
\begin{equation*}
\begin{aligned}
\sum_{0\leq k\in\Z}\|P_ku\|_{L^\infty}\lesssim t^{-\frac{1}{2}} .
\end{aligned}
\end{equation*}
The desired bound for \( \sum_{2^k\leq t^{-1}}\|P_k u\|_{L^\infty}\) may be obtained as before. We now consider the frequency regime  \(1\geq 2^k\geq t^{-1}\).  It follows from \eqref{9} that
  	\begin{equation*}
  	\begin{aligned}
  	\sum_{1\geq 2^k\geq t^{-1}}\left\|P_ku\right\|_{L^\infty}
  	&\lesssim t^{-\frac{1}{2}}\sum_{1\geq 2^k\geq t^{-1}}\frac{2^{\frac{1-\alpha}{2}k}}{1+2^{10k}}
  	\big\|(1+2^{10k})\widehat{P_k^\prime f}\big\|_{L^\infty}\\
  	&\quad+t^{-\frac{3}{4}}t^{p_0}\sum_{1\geq 2^k\geq t^{-1}}2^{-\frac{1+3\alpha}{4}k}\\
  	&\lesssim t^{-\frac{1}{2}}+t^{-\frac{3}{4}}t^{\frac{1+3\alpha}{4}}t^{p_0}\lesssim t^{-\frac{1}{2}},
  	\end{aligned}
  	\end{equation*}
where we have used the assumption \(p_0\leq-\frac{\alpha}{100}\) in the last inequality.  	

The desired bound for \(\|\partial_xu\|_{L^\infty}\) can be estimated as before.

\end{proof}

\section{Estimates on \(\|u\|_{H^{N_0}}\) and \(\|f\|_{H^{1,1}}\)}
In this section we will prove uniform bounds for the energy part in \eqref{4}.  More precisely :
\begin{theorem}\label{pr:1} Let \(u\) be a solution of \eqref{eq:main}-\eqref{eq:initial} satisfying the a priori bounds \eqref{4}. Then the following estimates hold true:
	\begin{align}\label{10}
	\|u(t,\cdot)\|_{H^{N_0}}\leq C\varepsilon_0\langle t \rangle^{C\varepsilon_1^2},
	\end{align}
and
	\begin{align}\label{11}
	\|f(t,\cdot)\|_{H^{1,1}}\leq C(\varepsilon_0+ \varepsilon_1^3)\langle t \rangle^{C\varepsilon_1^2}.
	\end{align}	
\end{theorem}

\begin{proof} Let 
	\begin{align*}
	\mathcal{L}=\partial_t-|D|^\alpha \partial_x,\  \mathcal{J}=x+(\alpha+1)t|D|^\alpha,
\  \Lambda=\partial_x^{-1}\big((\alpha+1)t\partial_t+x\partial_x+1\big).
	\end{align*}
Since \([\mathcal{L},\partial_x]=0\), the solution \(u\) of  \eqref{eq:main}-\eqref{eq:initial} satisfies 
\begin{align*}
\mathcal{L}\partial_x^ku=-\partial_x^k(u^2\partial_xu),\quad k=1,2,\dots,N_0.
\end{align*}	
This gives the energy identity 
\begin{align*}
\frac{1}{2}\frac{\diff}{\diff t}\int|\partial_x^ku|^2\,\diff x
=-\int\big(\partial_x^k(u^2\partial_xu)-u^2\partial_x^{k+1}u\big)\partial_x^ku\,\diff x-\int u^2\partial_x^{k+1}u\partial_x^ku\,\diff x.
\end{align*}	
Notice that 
\begin{align*}
\|\partial_x^k(u^2\partial_xu)-u^2\partial_x^{k+1}u\|_{L^2}
&\lesssim \|\partial_x(u^2)\|_{L^\infty}\|\partial_x^ku\|_{L^2}+\|\partial_xu\|_{L^\infty}\|\partial_x^k(u^2)\|_{L^2}\\
&\lesssim \|u\|_{L^\infty}\|\partial_xu\|_{L^\infty}\|\partial_x^ku\|_{L^2},
\end{align*}
and
\begin{align*}
\int u^2\partial_x^{k+1}u\partial_x^ku\,\diff x
=-\frac{1}{2}\int \partial_x(u^2)|\partial_x^ku|^2\,\diff x
\lesssim \|u\|_{L^\infty}\|\partial_xu\|_{L^\infty}\|\partial_x^ku\|_{L^2}^2.
\end{align*}
Hence 
\begin{align*}
\frac{\diff}{\diff t}\|\partial_x^ku\|_{L^2}^2
\lesssim \|u\|_{L^\infty}\|\partial_xu\|_{L^\infty}\|\partial_x^ku\|_{L^2}^2
\lesssim \varepsilon_1^2\langle t \rangle^{-1}\|\partial_x^ku\|_{L^2}^2,
\end{align*}
where we have used \eqref{7}, which entails \eqref{10}.

A small calculation shows
\begin{align*}
\mathcal{L}\Lambda u=u^2\partial_x\Lambda u+(\alpha-2)u^3/3.
\end{align*}
It follows that 
\begin{equation*}
\begin{aligned}
\frac{1}{2}\frac{\diff}{\diff t}\|\Lambda u\|_{L^2}^2&=\int u^2\partial_x\Lambda u \Lambda u\,\diff x+\frac{\alpha-2}{3}\int u^3\Lambda u\,\diff x\\
&\lesssim \|u\|_{L^\infty}\|\partial_x u\|_{L^\infty}\|\Lambda u\|_{L^2}^2+\|u\|_{L^\infty}^2\|u\|_{L^2}\|\Lambda u\|_{L^2},
\end{aligned}
\end{equation*}
which combined with  \eqref{7}  lead to
\begin{align}\label{12}
\|\Lambda u\|_{L^2}\lesssim  \varepsilon_0 \langle t \rangle^{C\varepsilon_1^2}.
\end{align}
Notice that
\begin{align}\label{13}
\mathcal{J}u=\Lambda u+(\alpha+1)t u^3/3,
\end{align}
then we employ \eqref{7}, \eqref{12} and \eqref{13} to deduce
\begin{equation}\label{13.5}
\begin{aligned}
\|x f\|_{L^2}=\|\mathcal{J}u\|_{L^2}\lesssim  \varepsilon_0 \langle t \rangle^{C\varepsilon_1^2}+t\|u\|_{L^\infty}^2\|u\|_{L^2}\lesssim (\varepsilon_0+ \varepsilon_1^3) \langle t \rangle^{C\varepsilon_1^2}.
\end{aligned}
\end{equation}

Using \([\mathcal{L},\partial_x]=0\) and \([\Lambda,\partial_x]=-\mathrm{id}\), one calculates that
\begin{align*}
\mathcal{L}\Lambda\partial_x u
=u^2\partial_x\Lambda \partial_x u+2u\partial_x u\Lambda \partial_x u+(\alpha+2) u^2\partial_x u.
\end{align*}
This results into the estimate
\begin{align*}
\frac{\diff}{\diff t}\|\Lambda \partial_xu\|_{L^2}^2
\lesssim \|u\|_{L^\infty}\|\partial_xu\|_{L^\infty}\|\Lambda \partial_xu\|_{L^2}^2+\|u\|_{L^\infty}\|\partial_xu\|_{L^\infty}\|u\|_{L^2}\|\Lambda \partial_xu\|_{L^2},
\end{align*}
which yields
\begin{align}\label{14}
\|\Lambda \partial_x u\|_{L^2}\lesssim \varepsilon_0 \langle t \rangle^{C\varepsilon_1^2}.
\end{align}
We use \([\mathcal{J},\partial_x]=-\mathrm{id}\) to calculate 
\begin{align}\label{15}
	\mathcal{J}\partial_x u
	=-\Lambda u-(\alpha+1)t u^3/3+(\Lambda\partial_x+1) u+(\alpha+1)t u^2\partial_x u.
\end{align}
Therefore the estimates \eqref{7}, \eqref{10}, \eqref{12} and \eqref{14} yield
\begin{align}\label{16}
\|x\partial_x f\|_{L^2}=\|\mathcal{J} \partial_x u\|_{L^2}\lesssim (\varepsilon_0+ \varepsilon_1^3) \langle t \rangle^{C\varepsilon_1^2}.
\end{align}

The estimate \eqref{11} is a consequence of \eqref{13.5} and \eqref{16}.

\end{proof}

\section{Estimates on \(\|f\|_Z\)}
Let
\begin{equation*}
\begin{aligned}
\Phi(\xi,\eta,\sigma)=|\xi|^{\alpha}\xi-|\xi-\eta-\sigma|^{\alpha}(\xi-\eta-\sigma)-|\eta|^{\alpha}\eta-|\sigma|^{\alpha}\sigma.
\end{aligned}
\end{equation*}
Taking the Fourier transform of \eqref{eq:main} gives
\begin{equation}\label{17}
\begin{aligned}
\partial_t\widehat{f}(\xi,t)&=-\mathrm{i}(2\pi)^{-1}\xi\int_{\R^2}e^{\mathrm{i}t\Phi(\xi,\eta,\sigma)}
\widehat{f}(\xi-\eta-\sigma,t)\widehat{f}(\eta,t)\widehat{f}(\sigma,t)\,\diff\eta \diff \sigma\\
&=\colon-\mathrm{i}(2\pi)^{-1}I(\xi,t).
\end{aligned}
\end{equation}
Set
\begin{align*}
	H(\xi,t):=\frac{3\xi|\xi|^{1-\alpha}}{\alpha(\alpha+1)}\int_1^t|\widehat{f}(s,\xi)|^2\frac{\diff s}{s},
\end{align*}	
and
\begin{align*}
g(\xi,t):=e^{\mathrm{i}H(\xi,t)}\widehat{f}(\xi,t).
\end{align*}
Then \(g\) satisfies the following evolutionary equation
\begin{align}\label{18}
\partial_t g(\xi,t)=-\mathrm{i}(2\pi)^{-1}e^{\mathrm{i}H(\xi,t)}\big( I(\xi,t)-3\tilde{c}t^{-1}\xi|\xi|^{1-\alpha}|\widehat{f}(\xi,t)|^2\widehat{f}(\xi,t)\big),
\end{align}
where \(\tilde{c}:=2\pi/[\alpha(\alpha+1)]\).

This section is  aimed to show the following theorem:
\begin{theorem}\label{pr:2}
It holds that
\begin{align}\label{19}
t_1^{p_0}\left\|(1+|\xi|)^{10}\big(g(\xi,t_2)-g(\xi,t_1)\big)\right\|_{L^\infty_\xi}\lesssim\epsilon_0,
\end{align} 
for any \(t_1\leq t_2\in[1,T]\). 
\end{theorem}
\noindent  The \(Z\)-norm part of \eqref{5} is an immediate consequence of the estimate \eqref{19} which also entails \eqref{3}.

\subsection{Proof of Theorem \ref{pr:2}}
We decompose in frequencies 
\begin{equation}\label{20}
\begin{aligned}
I(\xi,t)=\sum_{k_1,k_2,k_3\in \Z}I_{k_1,k_2,k_3}(\xi,t),
\end{aligned}
\end{equation}
in which  
\begin{equation}\label{21}
\begin{aligned}
I_{k_1,k_2,k_3}(\xi,t):=\xi\int_{\R^2}e^{\mathrm{i}t\Phi(\xi,\eta,\sigma)}\widehat{f_{k_1}}(\xi-\eta-\sigma,t)\widehat{f_{k_2}}(\eta,t)\widehat{f_{k_3}}(\sigma,t)\,\diff\eta \diff \sigma.
\end{aligned}
\end{equation}

For \eqref{19}, it suffices to show that if \(t_1\leq t_2\in [2^m-1,2^{m+1}]\cap[1,T]\), for some \(m\in \{1,2,\dots\}\), then
\begin{align}\label{22}
\left\|(1+|\xi|)^{10}\big(g(\xi,t_2)-g(\xi,t_1)\big)\right\|_{L^\infty_\xi}\lesssim\epsilon_02^{-p_0m}.
\end{align} 

Let \(k\in\Z\) and \(|\xi|\in[2^k,2^{k+1}]\) and \(s\in[2^m-1,2^{m+1}]\cap[1,T]\). Using the interpolation \eqref{75}, similar to \eqref{8.5}, it is easy to see that
\begin{equation}\label{24}
	\begin{aligned}
		\big|(1+|\xi|^{10})\widehat{f_k}(\xi)\big|\lesssim \epsilon_02^{-p_0m},
	\end{aligned}
\end{equation}
for  \(k\in [p_0m,\infty)\cap \Z\), which entails \eqref{22} in this frequency regime.

By the assumptions \eqref{4}, for \(s\in[2^m-1,2^{m+1}]\cap[1,T]\) and any \(l\in\Z\), we have 
\begin{equation}\label{23}
\begin{aligned}
&\|\widehat{f_l}(s)\|_{L^2}\lesssim \epsilon_12^{p_0m}2^{-N_0l_+},\\
&\|\partial\widehat{f_l}(s)\|_{L^2}\lesssim \epsilon_12^{p_0m}2^{-l},\\
&\|\widehat{f_l}(s)\|_{L^\infty}\lesssim \epsilon_12^{-10l_+},\\
&\|e^{\mp\mathrm{i}s\Lambda}f_l(s)\|_{L^\infty}\lesssim \epsilon_12^{-m/2},
\end{aligned}
\end{equation}
where \(l_+=\max(l,0)\).

It remains to consider \eqref{22} for \(|\xi|\in[2^k,2^{k+1}]\) with \(k\in (-\infty,p_0m]\).  By the equation \eqref{18} and the decomposition \eqref{20}-\eqref{21}, it suffices to show that 
\begin{equation}\label{25}
\begin{aligned}
&\sum_{k_1,k_2,k_3\in \Z}\bigg|\int_{t_1}^{t_2}e^{\mathrm{i}H(\xi,s)}\big( I_{k_1,k_2,k_3}(\xi,s)-\tilde{c}s^{-1}\xi|\xi|^{1-\alpha}\widehat{f_{k_1}}(\xi,s)\widehat{f_{k_2}}(\xi,s)\widehat{f_{k_3}}(-\xi,s)\\
&\quad\quad\quad\quad\quad\quad\quad\quad\quad\quad\quad\quad\quad\quad\quad-\tilde{c}
s^{-1}\xi|\xi|^{1-\alpha}\widehat{f_{k_1}}(\xi,s)\widehat{f_{k_2}}(-\xi,s)\widehat{f_{k_3}}(\xi,s)\\
&\quad\quad\quad\quad\quad\quad\quad\quad\quad\quad\quad\quad\quad-\tilde{c}s^{-1}\xi|\xi|^{1-\alpha}\widehat{f_{k_1}}(-\xi,s)\widehat{f_{k_2}}(\xi,s)\widehat{f_{k_3}}(\xi,s)\big)\,\diff s\bigg|\\
&\quad\quad\quad\quad\quad\quad\quad\quad\quad\quad\quad\quad\quad\quad\quad\lesssim
\epsilon_1^32^{-p_0m}2^{-10k_+},
\end{aligned}
\end{equation} 
for \(t_1\leq t_2\in [2^m-1,2^{m+1}]\cap [1,T]\). 

Using \eqref{23}, it is straightforward to show 
\begin{equation}\label{26}
\begin{aligned}
|I_{k_1,k_2,k_3}(\xi,s)|&\lesssim \epsilon_1^32^{3p_0m}2^k2^{\min(k_1,k_2,k_3)/2}2^{-N_0({k_1}_++{k_2}_++{k_3}_+)},
\end{aligned} 
\end{equation}
and
\begin{equation}\label{27}
\begin{aligned}
|I_{k_1,k_2,k_3}(\xi,s)|\lesssim \epsilon_1^32^k2^{\min(k_1,k_2,k_3)}2^{\text{med}(k_1,k_2,k_3)}2^{-10\max({k_1}_+,{k_2}_+,{k_3}_+)},
\end{aligned} 
\end{equation}
and
\begin{equation}\label{28}
\begin{aligned}
&\quad s^{-1}|\xi|^{2-\alpha}\big(\big|\widehat{f_{k_1}}(\xi,s)\widehat{f_{k_2}}(\xi,s)\widehat{f_{k_3}}(-\xi,s)\big|+\big|\widehat{f_{k_1}}(\xi,s)\widehat{f_{k_2}}(-\xi,s)\widehat{f_{k_3}}(\xi,s)\big|\\
&\quad\quad\quad\quad\quad\quad\quad\quad\quad\quad\quad\quad\quad+\big|\widehat{f_{k_1}}(-\xi,s)\widehat{f_{k_2}}(\xi,s)\widehat{f_{k_3}}(\xi,s)\big|\big)\\
&\lesssim \epsilon_1^32^{-m}2^{(2-\alpha)k}2^{-30k_+}{\bf{1}}_{[0,4]}\big(\max(|k_1-k|,|k_2-k|,|k_3-k|)\big).
\end{aligned} 
\end{equation}

With the bounds \eqref{26}-\eqref{28} at hand, one easily verifies \eqref{25} if one of the following conditions holds 
\begin{equation*}
\begin{aligned}
&\min(k_1,k_2,k_3)\leq -4m,\\
&\max(k_1,k_2,k_3)\geq p_0m/10,\\
&\min(k_1,k_2,k_3)+\mathrm{med}(k_1,k_2,k_3)\leq -(1+10p_0)m.
\end{aligned} 
\end{equation*}

Therefore, to complete the proof \eqref{25},  we are  reduced to show
\begin{theorem}\label{th:reduce}
Assume that \(k,k_1,k_2,k_3\in\Z\), \(m\in\Z\cap[100,\infty)\), \(|\xi|\in[2^k,2^{k+1}]\), and \(t_1\leq t_2\in [2^m-1,2^{m+1}]\cap [1,T]\). If
\begin{equation}\label{29}
\begin{aligned}
&k\in (-\infty,p_0m],\\
&k_1,k_2,k_3\in[-4m,p_0m/10],\\
&\min(k_1,k_2,k_3)+\mathrm{med}(k_1,k_2,k_3)\geq -(1+10p_0)m,
\end{aligned} 
\end{equation}
then 
\begin{equation}\label{30}
\begin{aligned}
&\bigg|\int_{t_1}^{t_2}e^{\mathrm{i}H(\xi,s)}\big( I_{k_1,k_2,k_3}(\xi,s)-\tilde{c}s^{-1}\xi|\xi|^{1-\alpha}\widehat{f_{k_1}}(\xi,s)\widehat{f_{k_2}}(\xi,s)\widehat{f_{k_3}}(-\xi,s)\\
&\quad\quad\quad\quad\quad\quad\quad\quad\quad\quad\quad-\tilde{c}
s^{-1}\xi|\xi|^{1-\alpha}\widehat{f_{k_1}}(\xi,s)\widehat{f_{k_2}}(-\xi,s)\widehat{f_{k_3}}(\xi,s)\\
&\quad\quad\quad\quad\quad\quad\quad\quad\quad\quad\quad-\tilde{c}s^{-1}\xi|\xi|^{1-\alpha}\widehat{f_{k_1}}(-\xi,s)\widehat{f_{k_2}}(\xi,s)\widehat{f_{k_3}}(\xi,s)\big)\,\diff s\bigg|\\
&\quad\quad\quad\quad\quad\quad\quad\quad\quad\quad\quad\lesssim\epsilon_1^32^{-2p_0m}2^{-10k_+}.
\end{aligned} 
\end{equation}
\end{theorem}

\subsection{Proof of Theorem \ref{th:reduce}}
We divide the proof into several propositions.
\begin{proposition}\label{le:1}
 Let \(k,k_1,k_2,k_3\in\Z\) and \(k_1,k_2,k_3\in[k-20,k+20]\), then the bound \eqref{30} holds.	
\end{proposition}
 It is easy to check \eqref{30} if \(k\leq -m/2\). Hence we will assume \(k\geq -m/2\) in the following. 
Without loss of generality, we may assume that \(\xi>0\) and \(\xi\in[2^k,2^{k+1}]\). We split the integral \(I_{k_1,k_2,k_3}\) as follows:
\begin{align*}
I_{k_1,k_2,k_3}(\xi,s)=\sum_{\iota_1,\iota_2,\iota_3\in\{+,-\}}I_{k_1,k_2,k_3}^{\iota_1,\iota_2,\iota_3}(\xi,s),
\end{align*}
with
\begin{align*}
I_{k_1,k_2,k_3}^{\iota_1,\iota_2,\iota_3}(\xi,s)=\xi\int_{\R^2}e^{it\Phi(\xi,\eta,\sigma)}\widehat{f_{k_1}^{\iota_1}}(\xi-\eta-\sigma,s)\widehat{f_{k_2}^{\iota_2}}(\eta,s)\widehat{f_{k_3}^{\iota_3}}(\sigma,s)\,\diff\eta\diff\sigma,
\end{align*}
where \(\widehat{f_{l}^{\iota}}(\mu):=\widehat{f_{l}}(\mu)1_{\iota}(\mu),1_+:=1_{[0,\infty)},1_-:=1_{(-\infty,0]}\). We first observe that \(I_{k_1,k_2,k_3}^{-,-,-}(\xi,s)=0\), so that  to finish the proof of \eqref{30} under the assumption of Proposition \ref{le:1}, we are then left to show the following two lemmas:
\begin{lemma}\label{le:2} It holds that
	\begin{equation}\label{31}
	\begin{aligned}
	&\quad\bigg|I_{k_1,k_2,k_3}^{+,+,-}(\xi,s)-\tilde{c}s^{-1}\xi|\xi|^{1-\alpha}\widehat{f_{k_1}}(\xi,s)\widehat{f_{k_2}}(\xi,s)\widehat{f_{k_3}}(-\xi,s)\bigg|\\
	&\quad+\bigg|I_{k_1,k_2,k_3}^{+,-,+}(\xi,s)-\tilde{c}s^{-1}\xi|\xi|^{1-\alpha}\widehat{f_{k_1}}(\xi,s)\widehat{f_{k_2}}(-\xi,s)\widehat{f_{k_3}}(\xi,s)\bigg|\\
	&\quad+\bigg|I_{k_1,k_2,k_3}^{-,+,+}(\xi,s)-\tilde{c}s^{-1}\xi|\xi|^{1-\alpha}\widehat{f_{k_1}}(-\xi,s)\widehat{f_{k_2}}(\xi,s)\widehat{f_{k_3}}(\xi,s)\bigg|\\
	&\lesssim \varepsilon_1^32^{-m}2^{-3p_0m}2^{-10k_+}.
	\end{aligned}
	\end{equation}
\end{lemma}
\begin{lemma}\label{le:3} 
	We have
	\begin{align}\label{32}
	\bigg|\int_{t_1}^{t_2}e^{\mathrm{i}H(s,\xi)}I_{k_1,k_2,k_3}^{\iota_1,\iota_2,\iota_3}(\xi,s)\,\diff s\bigg|
	\lesssim \varepsilon_1^32^{-2p_0m}2^{-10k_+},
	\end{align}
	for \((\iota_1,\iota_2,\iota_3)\in \{(+,+,+),(+,-,-),(-,+,-),(-,-,+)\}\).
\end{lemma}

\bigskip
We start by showing \eqref{31}.
\begin{proof}[Proof of Lemma \ref{le:2}] We only prove that the third term in LHS of \eqref{31} is bounded by \(\varepsilon_1^32^{-3p_1m}2^{-10k_+}\) and the other two terms may be handled similarly.  

Let \(\bar{l}\) be the smallest integer with the property that 
\begin{align*}
2^{\bar{l}}\geq 2^{(1-\alpha)k/2}2^{-49m/100}.
\end{align*}
Since \(k\geq -m/2\), one has \(\bar{l}\leq k-20\).  We may decompose 
\begin{align*}
I_{k_1,k_2,k_3}^{-,+,+}(\xi,s)=\sum_{l_1,l_2=\bar{l}}^{k+20}J_{l_1,l_2}(\xi,s),
\end{align*}
with
\begin{equation*}
\begin{aligned}
J_{l_1,l_2}(\xi,s)&=\xi\int_{\R^2}e^{\mathrm{i}s\Phi(\xi,\eta,\sigma)}\widehat{f_{k_1}^-}(\xi-\eta-\sigma,s)\widehat{f_{k_2}^+}(\eta,s)
\widehat{f_{k_3}^+}(\sigma,s)\\
&\quad\times\varphi_{l_1}^{(\bar{l})}(\xi-\eta)\varphi_{l_2}^{(\bar{l})}(\xi-\sigma)\,\diff\eta \diff \sigma,
\end{aligned}
\end{equation*}
for any \(l_1,l_2\geq \bar{l}\), and where 
\begin{equation*}
	\begin{aligned}
		\varphi_{k}^{(l)}(x):=\varphi(x/2^k),\  \text{if}\  k=l,
	\end{aligned}
\end{equation*}
and
\begin{equation*}
\begin{aligned}
\varphi_{k}^{(l)}(x):=\varphi(x/2^k)-\varphi(x/2^{k-1}),\  \text{if}\  k\geq l+1.
\end{aligned}
\end{equation*}

\bigskip
{\bf{Step 1: \(l_2\geq \max(l_1,\bar{l}+1)\) or \(l_1\geq \max(l_2,\bar{l}+1)\).}} We only consider the case \(l_2\geq \max(l_1,\bar{l}+1)\), a similar argument may apply to the other case. 

We will show
\begin{align*}
|J_{l_1,l_2}(\xi,s)|\lesssim \epsilon_1^32^{-m}2^{-3p_0m}2^{-10k_{+}}.
\end{align*}
On the support of the integral, one has \(|\xi-\eta-\sigma|\approx|\eta|\approx 2^k\) and \(|\xi-\sigma|\approx 2^{l_2}\), and then one finds 
\begin{equation}\label{33}
\begin{aligned}
|\partial_\eta\Phi(\xi,\eta,\sigma)|
=(\alpha+1)\big||\xi-\eta-\sigma|^{\alpha}-|\eta|^{\alpha}\big|
\gtrsim 2^{l_2}2^{(\alpha-1)k}.
\end{aligned}
\end{equation}
Using the identity
\begin{equation*}
	\begin{aligned}
	e^{\mathrm{i}s\Phi(\xi,\eta,\sigma)}=\frac{\partial_\eta e^{\mathrm{i}s\Phi(\xi,\eta,\sigma)}}{\mathrm{i}s\partial_\eta\Phi(\xi,\eta,\sigma)},
	\end{aligned}
	\end{equation*}
we use integration by parts in \(\eta\) to obtain
\begin{align*}
|J_{l_1,l_2}(\xi,s)|\leq |J_{l_1,l_2,1}(\xi,s)|+|F_{l_1,l_2,1}(\xi,s)|+|G_{l_1,l_2,1}(\xi,s)|,
\end{align*}
where 
\begin{equation*}
\begin{aligned}
J_{l_1,l_2,1}(\xi,s)&=\xi\int_{\R^2}\partial_\eta m_1(\eta,\sigma)e^{\mathrm{i}s\Phi(\xi,\eta,\sigma)}\widehat{f_{k_1}^-}(\xi-\eta-\sigma,s)\widehat{f_{k_2}^+}(\eta,s)
\widehat{f_{k_3}^+}(\sigma,s)\, \diff \eta\diff \sigma,\\
F_{l_1,l_1,1}(\xi,s)&=\xi\int_{\R^2}m_1(\eta,\sigma)e^{\mathrm{i}s\Phi(\xi,\eta,\sigma)}\partial_\eta\widehat{f_{k_1}^-}(\xi-\eta-\sigma,s)\widehat{f_{k_2}^+}(\eta,s)
\widehat{f_{k_3}^+}(\sigma,s) \, \diff \eta\diff \sigma,\\
G_{l_1,l_1,1}(\xi,s)&=\xi\int_{\R^2}m_1(\eta,\sigma)e^{\mathrm{i}s\Phi(\xi,\eta,\sigma)}\widehat{f_{k_1}^-}(\xi-\eta-\sigma,s)\partial_\eta\widehat{f_{k_2}^+}(\eta,s)
\widehat{f_{k_3}^+}(\sigma,s) \, \diff \eta\diff \sigma,
\end{aligned}
\end{equation*}
with
\begin{align}\label{34}
m_1(\eta,\sigma):=\frac{\varphi_{l_1}^{(\bar{l})}(\xi-\eta)\varphi_{l_2}(\xi-\sigma)}{s\partial_\eta\Phi(\xi,\eta,\sigma)}
\varphi_{k_1}^\prime(\xi-\eta-\sigma)\varphi_{k_2}^\prime(\eta)\varphi_{k_3}^\prime(\sigma).
\end{align}

Following \eqref{33} and \eqref{34}, a straightforward calculation shows that 
\begin{equation*}
\begin{aligned}
|\partial_\eta^a\partial_\sigma^bm_1(\eta,\sigma)|&\lesssim (2^{-m}2^{-l_2}2^{(1-\alpha)k})(2^{-al_1}2^{-bl_2})\\
&\quad \times{\bf{1}}_{[0,2^{l_1+4}]}(|\xi-\eta|){\bf{1}}_{[2^{l_2-4},2^{l_2+4}]}(|\xi-\sigma|),
\end{aligned}
\end{equation*}
for \(a,b\in[0,20]\cap\Z\). Hence  
\begin{equation*}
\begin{aligned}
\|\mathcal{F}^{-1}(m_1)\|_{L^1}\lesssim 2^{-m}2^{-l_2}2^{(1-\alpha)k}.
\end{aligned}
\end{equation*}

We first estimate the term \(F_{l_1,l_1,1}\). 
Fix \(\xi\) and \(s\), let 
\begin{equation*}
\begin{aligned}
&\widehat{f}(\theta):=e^{-\mathrm{i}s|\xi-\theta|^\alpha(\xi-\theta)}\partial_\eta\widehat{f_{k_1}^-}(\xi-\theta,s),\\
&\widehat{g}(\eta):=e^{-\mathrm{i}s|\eta|^\alpha\eta}\widehat{f_{k_2}^+}(\eta,s),\\
&\widehat{h}(\sigma):=e^{-\mathrm{i}s|\sigma|^\alpha\sigma}\widehat{f_{k_3}^+}(\sigma,s).
\end{aligned}
\end{equation*}
which  in light of \eqref{23} gives
\begin{equation*}
\begin{aligned}
&\|f\|_{L^2}\lesssim \epsilon_12^{-k}2^{p_0m},\\
&\|g\|_{L^\infty}\lesssim \epsilon_12^{-m/2},\\
&\|h\|_{L^2}\lesssim \epsilon_12^{p_0m}2^{-N_0k_+}.
\end{aligned}
\end{equation*}
It then follows from Lemma \ref{symbol estimate} that
\begin{equation*}
\begin{aligned}
|F_{l_1,l_1,1}(\xi,s)|&\lesssim 2^k\|\mathcal{F}^{-1}(m_1)\|_{L^1}\|f\|_{L^2}\|g\|_{L^\infty}\|h\|_{L^2}\\
&\lesssim \epsilon_1^32^{-3m/2+2p_0m}2^{-l_2}2^{(1-\alpha)k}2^{-N_0k_+}\\
&\lesssim \epsilon_1^32^{-3m/2+2p_0m+49m/100}2^{(1-\alpha)k/2}2^{-N_0k_+}\\
&\lesssim \epsilon_1^32^{-m}2^{-m/200}2^{-10k_+},
\end{aligned}
\end{equation*}
which is stronger than what we need. A similar argument yields 
\begin{align*}
|G_{l_1,l_1,1}(\xi,s)|\lesssim \epsilon_1^32^{-m}2^{-m/200}2^{-10k_+}.
\end{align*}

To estimate the term \(J_{l_1,l_1,1}\), we integrate by parts in \(\eta\) again to deduce
\begin{align*}
|J_{l_1,l_2,1}(\xi,s)|\leq |J_{l_1,l_2,2}(\xi,s)|+|F_{l_1,l_2,2}(\xi,s)|+|G_{l_1,l_2,2}(\xi,s)|,
\end{align*}
in which 
\begin{equation*}
\begin{aligned}
J_{l_1,l_2,2}(\xi,s)&=\xi\int_{\R^2}\partial_\eta m_2(\eta,\sigma)e^{\mathrm{i}s\Phi(\xi,\eta,\sigma)}\widehat{f_{k_1}^-}(\xi-\eta-\sigma,s)\widehat{f_{k_2}^+}(\eta,s)
\widehat{f_{k_3}^+}(\sigma,s)\, \diff \eta\diff \sigma,\\
F_{l_1,l_2,2}(\xi,s)&=\xi\int_{\R^2}m_2(\eta,\sigma)e^{\mathrm{i}s\Phi(\xi,\eta,\sigma)}\partial_\eta\widehat{f_{k_1}^-}(\xi-\eta-\sigma,s)\widehat{f_{k_2}^+}(\eta,s)
\widehat{f_{k_3}^+}(\sigma,s) \, \diff \eta\diff \sigma,\\
G_{l_1,l_2,2}(\xi,s)&=\xi\int_{\R^2}m_2(\eta,\sigma)e^{\mathrm{i}s\Phi(\xi,\eta,\sigma)}\widehat{f_{k_1}^-}(\xi-\eta-\sigma,s)\partial_\eta\widehat{f_{k_2}^+}(\eta,s)
\widehat{f_{k_3}^+}(\sigma,s) \, \diff \eta\diff \sigma,
\end{aligned}
\end{equation*}
with
\begin{align}\label{35}
m_2(\eta,\sigma):=\frac{\partial_\eta m_1(\eta,\sigma)}{s\partial_\eta\Phi(\xi,\eta,\sigma)}.
\end{align}

It follows from \eqref{33} and \eqref{35} that \(m_2\) satisfies the following stronger estimate 
\begin{equation*}
\begin{aligned}
|\partial_\eta^a\partial_\sigma^bm_2(\eta,\sigma)|&\lesssim (2^{-m}2^{-l_1-l_2}2^{(1-\alpha)k})(2^{-m}2^{-l_2}2^{(1-\alpha)k})(2^{-al_1}2^{-bl_2})\\
&\quad\times{\bf{1}}_{[0,2^{l_1+4}]}(|\xi-\eta|){\bf{1}}_{[2^{l_2-4},2^{l_2+4}]}(|\xi-\sigma|),
\end{aligned}
\end{equation*}
for \(a,b\in[0,19]\cap\Z\).
In a similar fashion as \(F_{l_1,l_1,1}(\xi,s)\) and \(G_{l_1,l_1,1}(\xi,s)\), we employ Lemma \ref{symbol estimate} to obtain
\begin{equation*}
\begin{aligned}
|F_{l_1,l_1,2}(\xi,s)|+|G_{l_1,l_1,2}(\xi,s)|
\lesssim \epsilon_1^32^{-m}2^{-m/200}2^{-10k_+}.
\end{aligned}
\end{equation*}
We finally estimate the left term \(J_{l_1,l_1,2}(\xi,s)\) as follows:
 \begin{equation*}
 \begin{aligned}
 |J_{l_1,l_1,2}(\xi,s)|&\lesssim  2^k2^{l_1}2^{l_2}|\partial_\eta m_2|\big\|\widehat{f_{k_1}^-}\big\|_{L^\infty}\big\|\widehat{f_{k_2}^+}\big\|_{L^\infty}\big\|\widehat{f_{k_3}^+}\big\|_{L^\infty}\\
 &\lesssim \epsilon^32^{-m}2^{-l_1-l_2}2^{(1-\alpha)k}2^{-m}2^{(2-\alpha)k}2^{-30k_+}\\
 &\lesssim \epsilon^32^{-m}2^{-m/200}2^{-10k_+}.
 \end{aligned}
 \end{equation*}
 
\bigskip
{\bf{Step 2: \(l_1=l_2=\bar{l}\).}} In this case, it suffices to prove that
\begin{align}\label{36}
\big|J_{\bar{l},\bar{l}}(\xi,s)-\tilde{c}s^{-1}\xi|\xi|^{1-\alpha}\widehat{f_{k_1}^-}(-\xi,s)\widehat{f_{k_2}^+}(\xi,s)\widehat{f_{k_3}^+}(\xi,s)\big|
\lesssim \epsilon_1^32^{-m}2^{-2p_0m}2^{-10k_+}.
\end{align} 
Define
\begin{equation*}
\begin{aligned}
\tilde{J}_{\bar{l},\bar{l}}(\xi,s)&=\xi\int_{\R^2}e^{\mathrm{i}s\alpha(\alpha+1)(\xi-\eta)(\xi-\sigma)/|\xi|^{1-\alpha}}\widehat{f_{k_1}^-}(\xi-\eta-\sigma,s)\widehat{f_{k_2}^+}(\eta,s)
\widehat{f_{k_3}^+}(\sigma,s)\\
&\quad\times\varphi\big(2^{-\bar{l}}(\xi-\eta)\big)\varphi\big(2^{-\bar{l}}(\xi-\sigma)\big)\,\diff\eta \diff \sigma,
\end{aligned}
\end{equation*}
and observe
\begin{align*}
\Phi(\xi,\eta,\sigma)=\alpha(\alpha+1)|\xi|^{\alpha-1}(\xi-\eta)(\xi-\sigma)
+2^{(\alpha-2)k}\bigO\big[(|\xi-\eta|+|\xi-\sigma|)^3\big],
\end{align*} 
when \(|\xi-\eta|+|\xi-\sigma|\leq 2^{k-5}\), 
we then estimate
\begin{equation}\label{38}
\begin{aligned}
|J_{\bar{l},\bar{l}}(\xi,s)-\tilde{J}_{\bar{l},\bar{l}}(\xi,s)|
\lesssim \epsilon^32^m2^{(\alpha-1)k}2^{5\bar{l}}2^{-30k_+}
\lesssim \epsilon^32^{-m}2^{-2m/5}2^{-10k_+}.
\end{aligned}
\end{equation}
Notice that
\begin{equation*}
\begin{aligned}
\big|\widehat{f_{k_1}^+}(\xi+r,s)-\widehat{f_{k_1}^+}(\xi,s)\big|\lesssim 2^{\bar{l}/2}2^{-k}2^{p_0m},\quad \text{for}\ |r|\leq 2^{\bar{l}},
\end{aligned}
\end{equation*}
we thus obtain
\begin{equation*}
\begin{aligned}
&\big|\widehat{f_{k_1}^-}(\xi-\eta-\sigma,s)\widehat{f_{k_2}^+}(\eta,s)
\widehat{f_{k_3}^+}(\sigma,s)-\widehat{f_{k_1}^-}(-\xi,s)\widehat{f_{k_2}^+}(\xi,s)\widehat{f_{k_3}^+}(\xi,s)\big|\\
&\lesssim \epsilon^32^{\bar{l}/2}2^{p_0m}2^{-k}2^{-20k_+},\quad \text{for}\ |\xi-\eta|+|\xi-\sigma|\leq 2^{\bar{l}+4}.
\end{aligned}
\end{equation*}
It then follows that
\begin{equation}\label{39}
\begin{aligned}
&\bigg|\tilde{J}_{\bar{l},\bar{l}}(\xi,s)-\xi\int_{\R^2}e^{\mathrm{i}s\alpha(\alpha+1)(\xi-\eta)(\xi-\sigma)/|\xi|^{1-\alpha}}\widehat{f_{k_1}^-}(-\xi,s)\widehat{f_{k_2}^+}(\xi,s)\widehat{f_{k_3}^+}(\xi,s)\\
&\quad\quad\quad\quad\quad\quad\quad\quad\quad\quad\quad\quad\quad\times\varphi\big(2^{-\bar{l}}(\xi-\eta)\big)\varphi\big(2^{-\bar{l}}(\xi-\sigma)\big)\,\diff\eta \diff \sigma\bigg|\\
&\lesssim \epsilon^32^{\bar{l}/2}2^{p_0m}2^{2\bar{l}}2^{-20k_+}
\lesssim \epsilon^32^{-6m/5}2^{-10k_+}.
\end{aligned}
\end{equation}

One calculates  
\begin{equation*}
\begin{aligned}
\int_{\R^2} e^{-\mathrm{i}xy}e^{-x^2/N^2}e^{-y^2/N^2}\,\diff x \diff y
=\frac{2\pi N}{\sqrt{4N^{-2}+N^2}}=2\pi+\bigO(N^{-1}),
\end{aligned}
\end{equation*}
in which we have used the formula 
\begin{equation*}
\begin{aligned}
\int_\R e^{-ax^2-bx}\,\diff x=e^{b^2/(4a)}\sqrt{\pi/a}, \quad \mathrm{for}\ a,b\in \mathcal{C}, \Re a>0.
\end{aligned}
\end{equation*}
It  follows that
\begin{equation*}
\begin{aligned}
\int_{\R^2} e^{-\mathrm{i}xy}\varphi(x/N)\varphi(y/N)\,\diff x \diff y
=2\pi+\bigO(N^{-1/2}),\quad \text{for}\ N\geq 1.
\end{aligned}
\end{equation*}
We then may estimate 
\begin{equation*}
\begin{aligned}
&\left|\int_{\R^2}e^{\mathrm{i}s\alpha(\alpha+1)(\xi-\eta)(\xi-\sigma)/|\xi|^{1-\alpha}}\varphi\big(2^{-\bar{l}}(\xi-\eta)\big)\varphi\big(2^{-\bar{l}}(\xi-\sigma)\big)\,\diff\eta \diff \sigma+\frac{2\pi|\xi|^{1-\alpha}}{s\alpha(\alpha+1)}\right|\\
&\lesssim 2^{(1-\alpha)k}2^{-m}\big(2^{m/100}\big)^{-1/2}
\lesssim 2^{-m}2^{-m/300}.
\end{aligned}
\end{equation*}
Hence
\begin{equation}\label{40}
\begin{aligned}
&\bigg|\xi\int_{\R^2}e^{\mathrm{i}s\alpha(\alpha+1)(\xi-\eta)(\xi-\sigma)/|\xi|^{1-\alpha}}\widehat{f_{k_1}^-}(-\xi,s)\widehat{f_{k_2}^+}(\xi,s)\widehat{f_{k_3}^+}(\xi,s)\varphi\big(2^{-\bar{l}}(\xi-\eta)\big)\\
&\quad\times\varphi\big(2^{-\bar{l}}(\xi-\sigma)\big)\,\diff\eta \diff \sigma+\frac{2\pi \xi|\xi|^{1-\alpha}}{s\alpha(\alpha+1)}\widehat{f_{k_1}^-}(-\xi,s)\widehat{f_{k_2}^+}(\xi,s)\widehat{f_{k_3}^+}(\xi,s)\bigg|\\
&\lesssim \epsilon^32^k2^{-m}2^{-m/300}2^{-30k_+}
\lesssim \epsilon^32^{-m}2^{-m/300}2^{-10k_+}.
\end{aligned}
\end{equation}

We finally conclude \eqref{36} from  \eqref{38}-\eqref{40}.
\end{proof}

\begin{proof}[Proof of Lemma \ref{le:3}]   
	Since \(\alpha\in(-1,0)\), it is straightforward to check
	\begin{align}\label{41}
	a^{\alpha+1}+b^{\alpha+1}+c^{\alpha+1}-(a+b+c)^{\alpha+1}\gtrsim b^{\alpha+1},
	\end{align}
if \(a\geq b\geq c\in(0,\infty)\). We only present the proof of the case \((\iota_1,\iota_2,\iota_3)=(+,-,-)\), since the other cases can be handled in a similar fashion.
	Use \eqref{41} and recall \(k_1,k_2,k_3\in[k-20,k+20]\), we then have 
	\begin{equation}\label{42}
	\begin{aligned}
	|\Phi(\xi,\eta,\sigma)|\gtrsim 2^{(\alpha+1)k}.
	\end{aligned}
	\end{equation}
	Observing
	\begin{equation*}
	\begin{aligned}
	e^{\mathrm{i}s\Phi(\xi,\eta,\sigma)}=\frac{\partial_se^{\mathrm{i}s\Phi(\xi,\eta,\sigma)}}{\mathrm{i}\Phi(\xi,\eta,\sigma)},
	\end{aligned}
	\end{equation*}
	we integrate by parts in \(s\) to deduce
	\begin{equation}\label{43}
	\begin{aligned}
	&\quad\bigg|\xi\int_{t_1}^{t_2}\int_{\R^2}e^{\mathrm{i}H(\xi,s)}e^{\mathrm{i}s\Phi(\xi,\eta,\sigma)}\widehat{f_{k_1}^{\iota_1}}(\xi-\eta-\sigma,s)\widehat{f_{k_2}^{\iota_2}}(\eta,s)\widehat{f_{k_3}^{\iota_3}}(\sigma,s)
	\,\diff\eta \diff \sigma\diff s\bigg|\\
	&\lesssim  \int_{t_1}^{t_2}\bigg|\xi\int_{\R^2}\frac{e^{\mathrm{i}s\Phi(\xi,\eta,\sigma)}}{\Phi(\xi,\eta,\sigma)}\frac{\diff}{\diff s}\bigg[e^{\mathrm{i}H(\xi,s)}\widehat{f_{k_1}^{\iota_1}}(\xi-\eta-\sigma,s)\widehat{f_{k_2}^{\iota_2}}(\eta,s)\widehat{f_{k_3}^{\iota_3}}(\sigma,s)\bigg]\,\diff\eta \diff \sigma \bigg|\diff s\\
	&\quad+ \sum_{j=1}^2\bigg|\xi\int_{\R^2}\frac{e^{\mathrm{i}H(\xi,s_j)}e^{\mathrm{i}s_j\Phi(\xi,\eta,\sigma)}}{\Phi(\xi,\eta,\sigma)}\widehat{f_{k_1}
		^{\iota_1}}(\xi-\eta-\sigma,s_j)\widehat{f_{k_2}^{\iota_2}}(\eta,s_j)\widehat{f_{k_3}^{\iota_3}}(\sigma,s_j)\,\diff\eta \diff \sigma\bigg|\\
	&=\colon A^0(\xi,s)+\sum_{j=1}^2A_j(\xi,s_j).
	\end{aligned}
	\end{equation} 
	
	We first handle the two terms \(A_1(\xi,s)\) and \(A_2(\xi,s)\).
	Let 
	\begin{equation}\label{44}
	\begin{aligned}
	m_3(\eta,\sigma):=\frac{1}{\Phi(\xi,\eta,\sigma)}\varphi_{k_1}^\prime(\xi-\eta-\sigma)\varphi_{k_2}^\prime(\eta)\varphi_{k_3}^\prime(\sigma).
	\end{aligned}
	\end{equation} 
	It thus follows from \eqref{42} and \eqref{44} that
	\begin{equation*}
	\begin{aligned}
	\|\mathcal{F}^{-1}(m_3)\|_{L^1}\lesssim 2^{-(\alpha+1)k}.
	\end{aligned}
	\end{equation*}
	We therefore, in light of Lemma \ref{symbol estimate}, may estimate
	\begin{equation*}
	\begin{aligned}
	\sup_{s\in[t_1,t_2]}A_j(\xi,s_j)
	\lesssim \varepsilon_1^32^{-\alpha k}2^{(2p_0-1/2)m}2^{-2N_0k_+}
	\lesssim \epsilon_1^32^{-m/4}2^{-10k_+},
	\end{aligned}
	\end{equation*}
	for \(j=1,2\), where we have used \(\alpha\in(-1,0)\) in the last inequality.

	We now come to estimate the term \(A^0(\xi,s)\). For this, we expand \(\diff /\diff s\) to deduce	
	\begin{equation}\label{45}
	\begin{aligned}
	A^0(\xi,s)\lesssim 2^m\sup_{s\in[t_1,t_2]}\big(A_0^0(\xi,s)+A_1^0(\xi,s)+A_2^0(\xi,s)+A_3^0(\xi,s)\big),
	\end{aligned}
	\end{equation}
	where
	\begin{equation}\label{46}
	\begin{aligned}
	A_0^0(\xi,s)&=\bigg|\xi\int_{\R^2}m_4(\eta,\sigma)e^{\mathrm{i}s\Phi(\xi,\eta,\sigma)}\widehat{f_{k_1}^{\iota_1}}(\xi-\eta-\sigma,s)\widehat{f_{k_2}^{\iota_2}}(\eta,s)\widehat{f_{k_3}^{\iota_3}}(\sigma,s)\,\diff\eta \diff \sigma\bigg|,\\
	A_1^0(\xi,s)&= \bigg|\xi\int_{\R^2}m_3(\eta,\sigma)e^{\mathrm{i}s\Phi(\xi,\eta,\sigma)}\partial_s\widehat{f_{k_1}^{\iota_1}}(\xi-\eta-\sigma,s)\widehat{f_{k_2}^{\iota_2}}(\eta,s)\widehat{f_{k_3}^{\iota_3}}(\sigma,s)\,\diff\eta \diff \sigma\bigg|,\\
	A_2^0(\xi,s)&= \bigg|\xi\int_{\R^2}m_3(\eta,\sigma)e^{\mathrm{i}s\Phi(\xi,\eta,\sigma)}\widehat{f_{k_1}^{\iota_1}}(\xi-\eta-\sigma,s)\partial_s\widehat{f_{k_2}^{\iota_2}}(\eta,s)\widehat{f_{k_3}^{\iota_3}}(\sigma,s)\,\diff\eta \diff \sigma\bigg|,\\
	A_3^0(\xi,s)&= \bigg|\xi\int_{\R^2}m_3(\eta,\sigma)e^{\mathrm{i}s\Phi(\xi,\eta,\sigma)}\widehat{f_{k_1}^{\iota_1}}(\xi-\eta-\sigma,s)\widehat{f_{k_2}^{\iota_2}}(\eta,s)\partial_s\widehat{f_{k_3}^{\iota_3}}(\sigma,s)\,\diff\eta \diff \sigma\bigg|,
	\end{aligned}
	\end{equation} 
	with 
	\begin{equation}\label{47}
	\begin{aligned}
	m_4(\eta,\sigma):=\partial_s H(\xi,s)m_3(\eta,\sigma).
	\end{aligned}
	\end{equation}
	
	To control \(A_0^0(\xi,s)\), one observes
	\begin{equation}\label{48}
	\begin{aligned}
	|\partial_s H(\xi,s)|
	\lesssim \epsilon_1^22^{(2-\alpha)k}2^{-m}2^{-20k_+},
	\end{aligned}
	\end{equation} 
	and thus obtains
	\begin{equation*}
	\begin{aligned}
	\|\mathcal{F}^{-1}(m_4)\|_{L^1}\lesssim \epsilon_1^22^{(1-2\alpha)k}2^{-m}2^{-20k_+}.
	\end{aligned}
	\end{equation*}
	We apply Lemma \ref{symbol estimate} again to obtain
	\begin{equation*}
	\begin{aligned}
	\sup_{s\in[t_1,t_2]}A_0^0(\xi,s)
	&\lesssim \epsilon_1^52^{(2-2\alpha)k}2^{(2p_0-3/2)m}2^{-2N_0k_+}2^{-20k_+}\\
	&\lesssim \epsilon_1^52^{-m}2^{-m/4}2^{-10k_+}.
	\end{aligned}
	\end{equation*}

	We next consider the term \(A_1^0(\xi,s)\). For this, in view of \eqref{17}, \eqref{20} and \eqref{23}, we first easily see that 
 \begin{align}\label{49}
	\|\partial_s\widehat{f_l}(s)\|_{L^2}\lesssim \epsilon_1^32^{-m}2^{3p_0m}2^{-20l_+}.
	\end{align}
Let
	\begin{equation*}
	\begin{aligned}
	&\widehat{f}(\theta):=e^{-\mathrm{i}s|\xi-\theta|^\alpha(\xi-\theta)}\partial_s\widehat{f_{k_1}^{\iota_1}}(\xi-\theta,s),\\
	&\widehat{g}(\eta):=e^{-\mathrm{i}s|\eta|^\alpha\eta}\widehat{f_{k_2}^{\iota_2}}(\eta,s),\\
	&\widehat{h}(\sigma):=e^{-\mathrm{i}s|\sigma|^\alpha\sigma}\widehat{f_{k_3}^{\iota_3}}(\sigma,s),
	\end{aligned}
	\end{equation*}
we then use \eqref{23} and \eqref{49} to get
	\begin{equation*}
	\begin{aligned}
	&\|f\|_{L^2}\lesssim \epsilon_12^{3p_0m}2^{-20k_+}2^{-m},\\ &\|g\|_{L^2}
	\lesssim \epsilon_12^{-N_0k_+}2^{p_0m},\\
	&\|h\|_{L^\infty}\lesssim \epsilon_12^{-m/2}.
	\end{aligned}
	\end{equation*}
	Recall \(\alpha\in(-1,0)\), it then follows from Lemma \ref{symbol estimate} that	\begin{equation*}
	\begin{aligned}
	\sup_{s\in[t_1,t_2]}A_1^0(\xi,s)
	&\lesssim \varepsilon_1^32^{-\alpha k}2^{(4p_0-3/2)m}2^{-N_0k_+}2^{-20k_+}\\
	&\lesssim \epsilon_1^32^{-m}2^{-m/4}2^{-10k_+}.
	\end{aligned}
	\end{equation*}
	Similarly, one has
	\begin{equation*}
	\begin{aligned}
	\sup_{s\in[t_1,t_2]}\big(A_2^0(\xi,s)+A_3^0(\xi,s)\big)
	\lesssim \epsilon_1^32^{-m}2^{-m/4}2^{-10k_+}.
	\end{aligned}
	\end{equation*}
		
\end{proof}

\begin{proposition}  Let
	\begin{equation}\label{49.5}
	\begin{aligned}
	&\min(k_1,k_2,k_3)\geq -\frac{(1-20p_0)m}{2(\alpha+1)},\\
	&\max(|k_1-k|,|k_2-k|,|k_3-k|)\geq 21,
	\end{aligned}
	\end{equation}
then the bound \eqref{30} holds.
\end{proposition}

\begin{proof}  It suffices to show
	\begin{equation*}
	\begin{aligned}
	| I_{k_1,k_2,k_3}(\xi,s)|
	\lesssim \epsilon_1^32^{-m}2^{-3p_0m}2^{-10k_+}.
	\end{aligned} 
	\end{equation*}

\bigskip
\noindent {\bf{Step 1: \(\max(|k_1-k_2|,|k_1-k_3|,|k_2-k_3|)\geq 5\).}}
Without loss of generality, we may assume that \(|k_1-k_2|\geq5\). On the support  \(|\xi-\eta-\sigma|\in [2^{k_1-2},2^{k_1+2}]\) and \(|\eta|\in [2^{k_2-2},2^{k_2+2}]\), one has
\begin{equation}\label{50}
\begin{aligned}
|\partial_\eta\Phi(\xi,\eta,\sigma)|
=(\alpha+1)\big||\xi-\eta-\sigma|^{\alpha}-|\eta|^{\alpha}\big|
\gtrsim 2^{\alpha\min(k_1,k_2)},
\end{aligned}
\end{equation}
and then integrate by parts in \(\eta\) to control
\begin{align}\label{51}
|I_{k_1,k_2,k_3}(\xi,s)|\leq |J_1(\xi,s)|+|F_1(\xi,s)|+|G_1(\xi,s)|,
\end{align}
where 
\begin{equation}\label{52}
\begin{aligned}
J_1(\xi,s)&=\xi\int_{\R^2}\partial_\eta m_5(\eta,\sigma)e^{\mathrm{i}s\Phi(\xi,\eta,\sigma)}\widehat{f_{k_1}}(\xi-\eta-\sigma,s)\widehat{f_{k_2}}(\eta,s)\widehat{f_{k_3}}(\sigma,s)\, \diff \eta\diff \sigma,\\
F_1(\xi,s)&=\xi\int_{\R^2}m_5(\eta,\sigma)e^{\mathrm{i}s\Phi(\xi,\eta,\sigma)}\partial_\eta\widehat{f_{k_1}}(\xi-\eta-\sigma,s)\widehat{f_{k_2}}(\eta,s)\widehat{f_{k_3}}(\sigma,s) \, \diff \eta\diff \sigma,\\
G_1(\xi,s)&=\xi\int_{\R^2}m_5(\eta,\sigma)e^{\mathrm{i}s\Phi(\xi,\eta,\sigma)}\widehat{f_{k_1}}(\xi-\eta-\sigma,s)\partial_\eta\widehat{f_{k_2}}(\eta,s)\widehat{f_{k_3}}(\sigma,s) \, \diff \eta\diff \sigma,
\end{aligned}
\end{equation}
with
\begin{align}\label{53}
m_5(\eta,\sigma):=\frac{1}{s\partial_\eta\Phi(\xi,\eta,\sigma)}\varphi_{k_1}^\prime(\xi-\eta-\sigma)\varphi_{k_2}^\prime(\eta)\varphi_{k_3}^\prime(\sigma).
\end{align}

Using \eqref{50} and \eqref{53}, one finds that \(m_5\) satisfies the following estimates
\begin{equation}\label{54}
\begin{aligned}
\|\mathcal{F}^{-1}(m_5)\|_{L^1}\lesssim 2^{-m}2^{-\alpha\min(k_1,k_2)},
\end{aligned}
\end{equation}
and
\begin{equation}\label{55}
\begin{aligned}
\|\mathcal{F}^{-1}(\partial_\eta m_5)\|_{L^1}\lesssim 2^{-m}2^{-(\alpha+1)\min(k_1,k_2)}.
\end{aligned}
\end{equation}

Applying Lemma \ref{symbol estimate}, we use \eqref{23} and \eqref{55} to see 
\begin{align*}
|J_1(\xi,s)|
\lesssim \epsilon_1^32^k2^{-(\alpha+1)\min(k_1,k_2)}2^{(2p_0-3/2)m}2^{-N_0\max({k_1}_+,{k_2}_+,{k_3}_+)},
\end{align*}
and instead use \eqref{23} and \eqref{54} to find
\begin{align*}
|F_1(\xi,s)|
\lesssim\epsilon_1^32^k2^{-k_1}2^{-\alpha\min(k_1,k_2)}2^{(2p_0-3/2)m}2^{-N_0\max({k_2}_+,{k_3}_+)},
\end{align*}
and
\begin{align*}
|G_1(\xi,s)|\lesssim\epsilon_1^32^k2^{-k_2}2^{-\alpha\min(k_1,k_2)}2^{(2p_0-3/2)m}2^{-N_0\max({k_1}_+,{k_3}_+)}.
\end{align*}
We finally conclude that
\begin{equation*}
\begin{aligned}
&\quad|J_1(\xi,s)|+|F_1(\xi,s)|+|G_1(\xi,s)|\\
&\lesssim \epsilon_1^32^k2^{-(\alpha+1)\min(k_1,k_2)}2^{(2p_0-3/2)m}2^{-10k_+}
(2^{10\max({k_1}_+,{k_2}_+,{k_3}_+)}+1)\\
&\lesssim\epsilon_1^32^{-m}2^{-5p_0m}2^{-10k_+},
\end{aligned}
\end{equation*}
where we have used \eqref{29} and \eqref{49.5} in the last inequality.

\bigskip
\noindent{\bf{Step 2: \(\max(|k_1-k_2|,|k_1-k_3|,|k_2-k_3|)\leq 4\).}}
In this case  \(\partial_\eta\Phi\neq0\) or \(\partial_\sigma\Phi\neq0\). Without loss of generality, we assume \(\partial_\eta\Phi\neq0\). On the support \(|\xi-\eta-\sigma|\in [2^{k_1-2},2^{k_1+2}]\), \(|\eta|\in [2^{k_2-2},2^{k_2+2}]\) and \(|\sigma|\approx 2^{k_3}\), recalling \(2^{k_1}\approx 2^{k_2}\approx 2^{k_3}\), we have
\begin{equation}\label{56}
\begin{aligned}
|\partial_\eta\Phi(\xi,\eta,\sigma)|
=(\alpha+1)\big||\xi-\eta-\sigma|^{\alpha}-|\eta|^{\alpha}\big|
\gtrsim 2^{\alpha k_2}.
\end{aligned}
\end{equation}
One may use integration by parts in \(\eta\) as \eqref{51}-\eqref{52} to control \(I_{k_1,k_2,k_3}\). Due to \eqref{56}, we instead have 
\begin{equation}\label{57}
\begin{aligned}
\|\mathcal{F}^{-1}(m_5)\|_{L^1}\lesssim 2^{-m}2^{-\alpha k_2},
\end{aligned}
\end{equation}
and
\begin{equation}\label{58}
\begin{aligned}
\|\mathcal{F}^{-1}(\partial_\eta m_5)\|_{L^1}\lesssim 2^{-m}2^{-(\alpha+1)k_2}.
\end{aligned}
\end{equation}
Recall \(2^{k_1}\approx 2^{k_2}\approx 2^{k_3}\) and repeat the arguments of {\bf{Step 2}}  using instead  \eqref{57}-\eqref{58}, we finally conclude that
\begin{equation*}
\begin{aligned}
|J_2(\xi,s)|+|F_2(\xi,s)|+|G_2(\xi,s)|
&\lesssim\epsilon_1^32^k2^{-(\alpha+1)k_2}2^{(2p_0-3/2)m}2^{-N_0{k_2}_+}\\
&\lesssim \epsilon_1^32^{-m}2^{-5p_0m}2^{-10k_+},
\end{aligned}
\end{equation*}
in which we have used the assumption
\begin{equation*}
\begin{aligned}
-(\alpha+1) k_2\leq (\frac{1}{2}-10p_0)m.
\end{aligned}
\end{equation*}
\end{proof}

\begin{proposition} Let
	\begin{equation}\label{58.5}
	\begin{aligned}
	&\min(k_1,k_2,k_3)\leq -\frac{(1-20p_0)m}{2(\alpha+1)},\\
	&\max(|k_1-k|,|k_2-k|,|k_3-k|)\geq 21,
	\end{aligned}
	\end{equation}
then the bound \eqref{30} holds.
\end{proposition}
\begin{proof} It suffices to show 
	\begin{align*}
	\bigg|\int_{t_1}^{t_2}e^{\mathrm{i}H(s,\xi)}I_{k_1,k_2,k_3}^{\iota_1,\iota_2,\iota_3}(\xi,s)\,\diff s\bigg|
	\lesssim \varepsilon_1^32^{-2p_0m}2^{-10k_+}.
	\end{align*}
	
\bigskip
\noindent{\bf{Case 1: \((\iota_1,\iota_2,\iota_3)= (+,+,+)\).}} In this case,  
one may estimate
\begin{equation*}
\begin{aligned}
-\Phi(\xi,\eta,\sigma)&=(\xi-\eta-\sigma)^{\alpha+1}+\eta^{\alpha+1}+\sigma^{\alpha+1}-\xi^{\alpha+1}\\
&\gtrsim 2^{(\alpha+1)\mathrm{med}(k_1,k_2,k_3)}.
\end{aligned}
\end{equation*}
It follows from \eqref{29} and \eqref{58.5} that
\begin{equation*}
\begin{aligned}
\mathrm{med}(k_1,k_2,k_3)\geq -(1+10p_0)m+\frac{(1-20p_0)m}{2(\alpha+1)}.
\end{aligned}
\end{equation*}
Therefore
\begin{equation}\label{59}
\begin{aligned}
-\Phi(\xi,\eta,\sigma)\gtrsim 2^{-(1+10p_0)(\alpha+1)m+(\frac{1}{2}-10p_0)m}\gtrsim 2^{-(\frac{1}{2}-20p_0)m}，
\end{aligned}
\end{equation}
where we have used the assumption \(p_0\leq-\frac{\alpha}{100}\).

Due to \eqref{59}, we may integrate by parts in \(s\) as \eqref{43}-\eqref{46} and need to control the terms \(A_j(\xi,s),A_l^0(\xi,s),j=1,2;l=0,1,2,3\).
Using \eqref{59} and \eqref{44}, one calculates that
\begin{equation}\label{60}
\begin{aligned}
\|\mathcal{F}^{-1}(m_3)\|_{L^1}\lesssim 2^{(\frac{1}{2}-20p_0)m},
\end{aligned}
\end{equation}
and
\begin{equation}\label{61}
\begin{aligned}
\|\mathcal{F}^{-1}(m_4)\|_{L^1}\lesssim \epsilon_1^22^{(\frac{1}{2}-20p_0)m}2^{(2-\alpha)k}2^{-m}2^{-20k_+}.
\end{aligned}
\end{equation}
The symbol-type estimate \eqref{60} together with  Lemma \ref{symbol estimate} yields 
\begin{equation*}
\begin{aligned}
\sup_{s\in[t_1,t_2]}A_j(\xi,s_j)
&\lesssim \epsilon_1^32^k2^{(\frac{1}{2}-20p_0)m}2^{(2p_0-1/2)m}2^{-N_0({k_1}_+,{k_2}_+,{k_3}_+)}\\
&\lesssim \epsilon_1^32^{-5p_0m}2^{-10k_+},\quad \text{for}\ j=1,2,
\end{aligned}
\end{equation*}
and
\begin{equation*}
\begin{aligned}
\sup_{s\in[t_1,t_2]}A_1^0(\xi,s)
&\lesssim \epsilon_1^32^k2^{(\frac{1}{2}-20p_0)m}2^{(4p_0-3/2)m}2^{-20{k_1}_+}2^{-N_0({k_2}_+,{k_3}_+)}\\
&\lesssim \epsilon_1^52^{-m}2^{-5p_0m}2^{-10k_+},
\end{aligned}
\end{equation*}
and the same bound for \(A_2^0(\xi,s)\) and \(A_3^0(\xi,s)\).
Applying Lemma \ref{symbol estimate} with \eqref{61}, we may estimate
\begin{equation*}
\begin{aligned}
\sup_{s\in[t_1,t_2]}A_0^0(\xi,s)
&\lesssim \epsilon_1^32^{(3-\alpha)k}2^{(\frac{1}{2}-20p_0)m}2^{(2p_0-3/2)m}2^{-N_0({k_1}_+,{k_2}_+,{k_3}_+)}2^{-20k_+}\\
&\lesssim \epsilon_1^52^{-m}2^{-5p_0m}2^{-10k_+}.
\end{aligned}
\end{equation*}

\bigskip
\noindent{\bf{Case 2: \((\iota_1,\iota_2,\iota_3)\in \{(+,-,-),(-,+,-),(-,-,+)\}\).}} Since they are similar, we only analyze the case \((\iota_1,\iota_2,\iota_3)= (+,-,-)\). 
In this case, we have 
\begin{equation*}
\begin{aligned}
\Phi(\xi,\eta,\sigma)
&=\xi^{\alpha+1}+(-\eta)^{\alpha+1}+(-\sigma)^{\alpha+1}-(\xi-\eta-\sigma)^{\alpha+1}\\
&\gtrsim 2^{(\alpha+1)\mathrm{med}(k,k_2,k_3)}.
\end{aligned}
\end{equation*}
There are three sub-cases to consider:\\
{\bf{(i)}} If \(\mathrm{med}(k,k_2,k_3)=k\), then 
\begin{equation*}
\begin{aligned}
\Phi(\xi,\eta,\sigma)\gtrsim 2^{(\alpha+1)k}.
\end{aligned}
\end{equation*}
{\bf{(ii)}} If \(\mathrm{med}(k,k_2,k_3)=k_2\), then \(k\leq k_2\leq k_3 \) or \(k_3\leq k_2\leq k\).\\ 
When \(k\leq k_2\leq k_3 \), one has 
\begin{equation*}
\begin{aligned}
\Phi(\xi,\eta,\sigma)\gtrsim 2^{(\alpha+1)k_2}\geq 2^{(\alpha+1)k}.
\end{aligned}
\end{equation*}
For \(k_3\leq k_2\leq k\), it holds 
\begin{equation*}
\begin{aligned}
\Phi(\xi,\eta,\sigma)\gtrsim 2^{(\alpha+1)k_2}\geq 2^{(\alpha+1)\mathrm{med}(k_1,k_2,k_3)}.
\end{aligned}
\end{equation*}
{\bf{(iii)}} If \(\mathrm{med}(k,k_2,k_3)=k_3\), then,  similarly to (ii), one may show \(\Phi\) enjoys the same bounds as (ii).

We conclude from (i)-(iii) that
\begin{equation}\label{62}
\begin{aligned}
\Phi(\xi,\eta,\sigma)\gtrsim 2^{(\alpha+1)k},
\end{aligned}
\end{equation}
or
\begin{equation}\label{63}
\begin{aligned}
\Phi(\xi,\eta,\sigma)\gtrsim 2^{(\alpha+1)\mathrm{med}(k_1,k_2,k_3)}\gtrsim 2^{-(\frac{1}{2}-20p_0)m}.
\end{aligned}
\end{equation}

For the case of \eqref{62}, the Phase \(\Phi\) enjoys the same bound as \eqref{42}, thus the terms \(A_j(\xi,s_j),A_l^0(\xi,s),j=1,2;l=0,1,2,3\) can be estimated as it in Lemma \ref{le:3}. 
The latter case \eqref{63} can be handled identically to {\bf{Case 1}}.

\bigskip
\noindent{\bf{Case 3: \((\iota_1,\iota_2,\iota_3)\in \{(+,+,-),(+,-,+),(-,+,+)\}\).}} We only consider the case \((\iota_1,\iota_2,\iota_3)=(+,+,-)\), and the other cases may be handled in a similar fashion.
In this case, one has 
\begin{equation*}
\begin{aligned}
\Phi(\xi,\eta,\sigma)=\xi^{\alpha+1}-(\xi-\eta-\sigma)^{\alpha+1}-\eta^{\alpha+1}+(-\sigma)^{\alpha+1}.
\end{aligned}
\end{equation*}

We shall divide it into two sub-cases.\\
{\bf{(i) \(k_3=\min(k_1,k_2,k_3)\).}} Recalling the assumption
\(
k_1,k_2,k_3\in[-4m,p_0m/10],
\)
we have 
\[
k_3\in \bigg[-(1+20p_0)m,-\frac{(1-20p_0)m}{2(\alpha+1)}\bigg],\quad k_1,k_2\in \bigg[-\frac{(1-40p_0)m}{2(\alpha+1)},\frac{p_0m}{10}\bigg].
\]
Notice that 
\begin{align*}
a^{\alpha+1}+b^{\alpha+1}-(a+b)^{\alpha+1}\gtrsim b^{\alpha+1},
\end{align*}
if \(a\geq b\in(0,\infty)\). 
We then may estimate 
\begin{equation}\label{66}
\begin{aligned}
-\Phi(\xi,\eta,\sigma)
&=\big(-\xi^{\alpha+1}+(\xi-\eta-\sigma)^{\alpha+1}
+(\eta+\sigma)^{\alpha+1}\big)\\
&\quad+\big(\eta^{\alpha+1}-(\eta+\sigma)^{\alpha+1}\big)
-(-\sigma)^{\alpha+1}\\
&\gtrsim 2^{(\alpha+1)\min(k_1,k_2)}\geq 2^{-(\frac{1}{2}-20p_0)m},
\end{aligned}
\end{equation}
provided that \(|\xi-\eta-\sigma|\in [2^{k_1-2},2^{k_1+2}]\), \(|\eta|\in [2^{k_2-2},2^{k_2+2}]\) and \(|\sigma|\in [2^{k_3-2},2^{k_3+2}]\).

Proceed as \eqref{43}-\eqref{46} by integration by parts in \(s\). It follows from \eqref{66} and \eqref{44} that
\begin{equation}\label{67}
\begin{aligned}
\|\mathcal{F}^{-1}(m_3)\|_{L^1}\lesssim 2^{(\frac{1}{2}-20p_0)m},
\end{aligned}
\end{equation}
and
\begin{equation}\label{68}
\begin{aligned}
\|\mathcal{F}^{-1}(m_4)\|_{L^1}\lesssim \epsilon_1^22^{(2-\alpha)k}2^{(\frac{1}{2}-20p_0)m}2^{-m}2^{-20k_+}.
\end{aligned}
\end{equation}
With the symbol-type bounds \eqref{67} and \eqref{68} at hand, repeating the argument of {\bf{Case 1}}, one may estimate
\begin{equation*}
\begin{aligned}
\sup_{s\in[t_1,t_2]}A_j(\xi,s)
\lesssim \epsilon_1^32^{-5p_0m}2^{-10k_+},\quad \text{for}\ j=1,2,
\end{aligned}
\end{equation*}
and
\begin{equation*}
\begin{aligned}
\sup_{s\in[t_1,t_2]}\big(A_0^0(\xi,s)+A_1^0(\xi,s)+A_2^0(\xi,s)+A_3^0(\xi,s)\big)\lesssim \epsilon_1^32^{-m}2^{-5p_0m}2^{-10k_+}.
\end{aligned}
\end{equation*}

\bigskip
\noindent {\bf{(ii) \(k_3\neq \min(k_1,k_2,k_3)\).}} By symmetry we may assume 
\(k_1=\min(k_1,k_2,k_3)\). 
In this case, we have 
\[
k_1\in \bigg[-(1+20p_0)m,-\frac{(1-20p_0)m}{2(\alpha+1)}\bigg],\quad k_2,k_3\in \bigg[-\frac{(1-40p_0)m}{2(\alpha+1)},\frac{p_0m}{10}\bigg].
\]

Define 
\begin{equation*}
\chi_{k,m}(\eta,\sigma)=
\left\{ 
\begin{aligned}
&1,&& \text{if}\ |k-k_1|\geq 11,\\
&1-\varphi\big(2^{(1+20p_0)m}(\eta+\sigma)\big),&& \text{if}\ |k-k_1|\leq 10.
\end{aligned}
\right.
\end{equation*}
We first observe that
\begin{equation*}
\begin{aligned}
&\quad\bigg|\xi\int_{t_1}^{t_2}\int_{\R^2}\big(1-\chi_{k,m}(\eta,\sigma)\big)e^{\mathrm{i}H(\xi,s)}e^{\mathrm{i}s\Phi(\xi,\eta,\sigma)}\widehat{f_{k_1}^+}(\xi-\eta-\sigma,s)\\
&\quad\quad\quad\quad\quad\quad\quad\quad\quad\quad\quad\quad\quad\quad\quad\times\widehat{f_{k_2}^+}(\eta,s)\widehat{f_{k_3}^-}(\sigma,s)\,\diff\eta \diff \sigma\diff s\bigg|\\
&\leq 2^m2^k\bigg|\int_{\R^2}\varphi\big(2^{(1+20p_0)m}(\eta+\sigma)\big)e^{\mathrm{i}s\Phi(\xi,\eta,\sigma)}\widehat{f_{k_1}^+}(\xi-\eta-\sigma,s)\\
&\quad\quad\quad\quad\quad\quad\quad\quad\quad\quad\quad\quad\quad\quad\quad\times\widehat{f_{k_2}^+}(\eta,s)\widehat{f_{k_3}^-}(\sigma,s)\,\diff\eta \diff \sigma\bigg|\\
&\lesssim \epsilon_1^32^m2^k2^{-(\frac{1}{2}+10p_0)m}2^{-m/2}2^{p_0m}2^{-10k_+}\\
&\lesssim\epsilon_1^32^{-5p_0m}2^{-10k_+}.
\end{aligned}
\end{equation*}
To complete the proof it remains to show
\begin{equation}\label{69}
\begin{aligned}
&\quad\bigg|\xi\int_{t_1}^{t_2}\int_{\R^2}\chi_{k,m}(\eta,\sigma)e^{\mathrm{i}H(\xi,s)}e^{\mathrm{i}s\Phi(\xi,\eta,\sigma)}\widehat{f_{k_1}^+}(\xi-\eta-\sigma,s)\\
&\quad\quad\quad\quad\quad\quad\quad\quad\quad\quad\quad\quad\quad\quad\quad\times\widehat{f_{k_2}^+}(\eta,s)\widehat{f_{k_3}^-}(\sigma,s)\,\diff\eta \diff \sigma\diff s\bigg|\\
&\lesssim\epsilon_1^32^{-5p_0m}2^{-10k_+}.
\end{aligned}
\end{equation}

The crucial ingredient in showing \eqref{69} 
is that the phase \(\Phi\) satisfies the following weakly elliptic bound
\begin{equation}\label{70}
|\Phi(\xi,\eta,\sigma)|\geq \lambda:=
\left\{ 
\begin{aligned}
&2^{(\alpha+1)k_1-100},&&\quad \text{if}\ k\leq k_1-11,\\
&2^{(\alpha+1)\min(k,k_2,k_3)-100},&&\quad \text{if}\ k\geq k_1+11,\\
&2^{-(1+20p_0)m}2^{\alpha k_1-100},&&\quad \text{if}\ |k-k_1|\leq 10,
\end{aligned}
\right.
\end{equation}
provided that \(|\xi-\eta-\sigma|\in [2^{k_1-2},2^{k_1+2}]\), \(|\eta|\in [2^{k_2-2},2^{k_2+2}]\), \(|\sigma|\in [2^{k_3-2},2^{k_3+2}]\) and \(\chi_{k,m}(\eta,\sigma)\neq 0\). We now show \eqref{70} as follows:

If \(k\leq k_1-11\), we then estimate 
\begin{equation*}
\begin{aligned}
-\Phi(\xi,\eta,\sigma)
&\geq -\xi^{\alpha+1}+(\xi-\eta-\sigma)^{\alpha+1}-|\eta^{\alpha+1}-(-\sigma)^{\alpha+1}|\\
&\geq -2^{(\alpha+1)k+1}+2^{(\alpha+1)k_1-1}-2^{k_1+\alpha k_2+10}\\
&\geq 2^{(\alpha+1)k_1-2}.
\end{aligned}
\end{equation*}

If \(k\geq k_1+11\), we then deduce
\begin{equation*}
\begin{aligned}
\quad\Phi(\xi,\eta,\sigma)
&\geq \big(\xi^{\alpha+1}+(-\sigma)^{\alpha+1}-(\xi-\sigma)^{\alpha+1}\big)-(\xi-\eta-\sigma)^{\alpha+1}\\
&\quad-|\eta^{\alpha+1}-(\xi-\sigma)^{\alpha+1}|\\
&\geq 2^{(\alpha+1)\min(k,k_3)-1}-2^{(\alpha+1)k_1+1}-2^{k_1+\alpha k_2+10}\\
&\geq 2^{(\alpha+1)\min(k,k_3)-2}.
\end{aligned}
\end{equation*}

If \(|k-k_1|\leq 10\), first notice \(2^{-(1+20p_0)m}\leq |\eta+\sigma|\leq 2^{k_1+11}\), then one has
\begin{equation*}
\begin{aligned}
|\Phi(\xi,\eta,\sigma)|&=(\alpha+1)\big|\big(\mu\xi+(1-\mu)(\xi-\eta-\sigma)\big)^{\alpha}-\big(\theta\eta+(1-\theta)(-\sigma)\big)^{\alpha}\big||\eta+\sigma|\\
&\gtrsim 2^{\alpha k_1}2^{-(1+20p_0)m},
\end{aligned}
\end{equation*}
for some \(\mu,\theta\in (0,1)\).

To show \eqref{69} in view of \eqref{70}, 
we integrate by parts in \(s\) to get
\begin{equation*}
\begin{aligned}
&\quad\bigg|\xi\int_{t_1}^{t_2}\int_{\R^2}\chi_{k,m}(\eta,\sigma)e^{\mathrm{i}H(\xi,s)}e^{\mathrm{i}s\Phi(\xi,\eta,\sigma)}\widehat{f_{k_1}^+}(\xi-\eta-\sigma,s)\\
&\quad\quad\quad\quad\quad\quad\quad\quad\quad\quad\quad\quad\quad\quad\quad\times\widehat{f_{k_2}^+}(\eta,s)\widehat{f_{k_3}^-}(\sigma,s)
\,\diff\eta \diff \sigma\diff s\bigg|\\
&\lesssim  \int_{t_1}^{t_2}\bigg|\xi\int_{\R^2}\chi_{k,m}(\eta,\sigma)\frac{e^{\mathrm{i}s\Phi(\xi,\eta,\sigma)}}{\Phi(\xi,\eta,\sigma)}\frac{\diff}{\diff s}\bigg[e^{\mathrm{i}H(\xi,s)}\widehat{f_{k_1}^+}(\xi-\eta-\sigma,s)\\
&\quad\quad\quad\quad\quad\quad\quad\quad\quad\quad\quad\quad\quad\quad\quad\times\widehat{f_{k_2}^+}(\eta,s)\widehat{f_{k_3}^-}(\sigma,s)\bigg]\,\diff\eta \diff \sigma \bigg|\diff s\\
&\quad+ \sum_{j=1}^2\bigg|\xi\int_{\R^2}\chi_{k,m}(\eta,\sigma)\frac{e^{\mathrm{i}H(\xi,s_j)}e^{\mathrm{i}s_j\Phi(\xi,\eta,\sigma)}}{\Phi(\xi,\eta,\sigma)}\widehat{f_{k_1}^+}(\xi-\eta-\sigma,s_j)\\
&\quad\quad\quad\quad\quad\quad\quad\quad\quad\quad\quad\quad\quad\quad\quad\times\widehat{f_{k_2}^+}(\eta,s_j)\widehat{f_{k_3}^-}(\sigma,s_j)\,\diff\eta \diff \sigma\bigg|\\
&=\colon B^0(\xi,s)+\sum_{j=1}^2B_j(\xi,s_j), 
\end{aligned}
\end{equation*} 
and the term \(B^0(\xi,s)\) may be bounded by
\begin{equation*}
\begin{aligned}
B^0(\xi,s)\lesssim 2^m\sup_{s\in[t_1,t_2]}\big(B_0^0(\xi,s)+B_1^0(\xi,s)+B_2^0(\xi,s)+B_3^0(\xi,s)\big),
\end{aligned}
\end{equation*}
in which each term is defined by
\begin{equation*}
\begin{aligned}
B_0^0(\xi,s)&=\bigg|\xi\int_{\R^2}m_7(\eta,\sigma)e^{\mathrm{i}s\Phi(\xi,\eta,\sigma)}\widehat{f_{k_1}^+}(\xi-\eta-\sigma,s)\widehat{f_{k_2}^+}(\eta,s)\widehat{f_{k_3}^-}(\sigma,s)\,\diff\eta \diff \sigma\bigg|,\\
B_1^0(\xi,s)&= \bigg|\xi\int_{\R^2}m_6(\eta,\sigma)e^{\mathrm{i}s\Phi(\xi,\eta,\sigma)}\partial_s\widehat{f_{k_1}^+}(\xi-\eta-\sigma,s)\widehat{f_{k_2}^+}(\eta,s)\widehat{f_{k_3}^-}(\sigma,s)\,\diff\eta \diff \sigma\bigg|,\\
B_2^0(\xi,s)&= \bigg|\xi\int_{\R^2}m_6(\eta,\sigma)e^{\mathrm{i}s\Phi(\xi,\eta,\sigma)}\widehat{f_{k_1}^+}(\xi-\eta-\sigma,s)\partial_s\widehat{f_{k_2}^+}(\eta,s)\widehat{f_{k_3}^-}(\sigma,s)\,\diff\eta \diff \sigma\bigg|,\\
B_3^0(\xi,s)&= \bigg|\xi\int_{\R^2}m_6(\eta,\sigma)e^{\mathrm{i}s\Phi(\xi,\eta,\sigma)}\widehat{f_{k_1}^+}(\xi-\eta-\sigma,s)\widehat{f_{k_2}^+}(\eta,s)\partial_s\widehat{f_{k_3}^-}(\sigma,s)\,\diff\eta \diff \sigma\bigg|,
\end{aligned}
\end{equation*} 
with
\begin{equation}\label{71}
\begin{aligned}
m_6(\eta,\sigma):=\frac{\chi_{k,m}(\eta,\sigma)}{\Phi(\xi,\eta,\sigma)}\varphi_{k_1}^\prime(\xi-\eta-\sigma)\varphi_{k_2}^\prime(\eta)\varphi_{k_3}^\prime(\sigma),
\end{aligned}
\end{equation} 
and
\begin{equation}\label{72}
\begin{aligned}
m_7(\eta,\sigma):=\chi_{k,m}(\eta,\sigma)\partial_s H(\xi,s)m_6(\eta,\sigma).
\end{aligned}
\end{equation}

It follows from \eqref{70} and \eqref{71} that
\begin{equation}\label{73}
\begin{aligned}
\|\mathcal{F}^{-1}(m_6)\|_{L^1}\lesssim \lambda^{-1},
\end{aligned}
\end{equation}
and from \eqref{70} and \eqref{72} that
\begin{equation}\label{74}
\begin{aligned}
\|\mathcal{F}^{-1}(m_7)\|_{L^1}\lesssim \epsilon_1^2\lambda^{-1}2^{(2-\alpha)k}2^{-m}2^{-20k_+}.
\end{aligned}
\end{equation}
Applying Lemma \ref{symbol estimate} with \eqref{73} yields
\begin{equation*}
\begin{aligned}
\sup_{s\in[t_1,t_2]}B_j(\xi,s_j)
\lesssim \epsilon_1^3\lambda^{-1}2^k2^{(2p_0-1/2)m}2^{-10k_+},\quad \text{for}\ j=1,2.
\end{aligned}
\end{equation*}
So we need to estimate \(\lambda^{-1}2^k\). If \(k\leq k_1-11\), then
\begin{equation*}
\begin{aligned}
\lambda^{-1}2^k
\lesssim 2^{-(\alpha+1)k_1}2^k
\lesssim 2^{-(\alpha+1)k}2^k
\lesssim 2^{-\alpha p_0m}.
\end{aligned}
\end{equation*}
If \(k\geq k_1+11\), then when \(k\leq\min(k_2,k_3)\)
\begin{equation*}
\begin{aligned}
\lambda^{-1}2^k
\lesssim 2^{-(\alpha+1)k}2^k
\lesssim 2^{-\alpha p_0m},
\end{aligned}
\end{equation*}
and when \(k\geq\min(k_2,k_3)\)
\begin{equation*}
\begin{aligned}
\lambda^{-1}2^k
\lesssim 2^{-(\alpha+1)\min(k_2,k_3)}2^k
\lesssim 2^{(\frac{1}{2}-20p_0)m}2^{p_0m}.
\end{aligned}
\end{equation*} 
If \(|k-k_1|\leq 10\), then
\begin{equation*}
\begin{aligned}
\lambda^{-1}2^k
\lesssim 2^{(1+20p_0)m}2^{(1-\alpha)k_1}
\lesssim 2^{(1+20p_0)m}2^{-\frac{(1-\alpha)(1-20p_0)m}{2(\alpha+1)}}
\lesssim 2^{(\frac{1}{2}-20p_0)m},
\end{aligned}
\end{equation*}
where we have used the assumption \(p_0\leq-\frac{\alpha}{100}\). 
We finally conclude
\begin{equation*}
\begin{aligned}
\sup_{s\in[t_1,t_2]}B_j(\xi,s_j)\lesssim \epsilon_1^32^{-5p_0m}2^{-10k_+},\quad \text{for}\ j=1,2.
\end{aligned}
\end{equation*}
Similarly, we obtain 
\begin{equation}
\begin{aligned}
\sup_{s\in[t_1,t_2]}\big(B_1^0(\xi,s)+B_2^0(\xi,s)+B_3^0(\xi,s)\big)
&\lesssim \epsilon_1^3\lambda^{-1}2^k2^{(4p_0-3/2)m}2^{-10k_+}\\
&\lesssim \epsilon_1^32^{-m}2^{-5p_0m}2^{-10k_+}.
\end{aligned}
\end{equation}
We apply Lemma \ref{symbol estimate} instead of using \eqref{74} to obtain
\begin{equation*}
\begin{aligned}
\sup_{s\in[t_1,t_2]}B_0^0(\xi,s)
&\lesssim \epsilon_1^5\lambda^{-1}2^{(3-\alpha)k}2^{(2p_0-3/2)m}2^{-10k_+}\\
&\lesssim \epsilon_1^52^{-m}2^{-5p_0m}2^{-10k_+}.
\end{aligned}
\end{equation*}

\end{proof}

\section{Concluding remarks}

We conclude this paper by some related and open questions.  Some previous results on long time behavior or finite time blow-up for solutions to fractional KdV type equations concerned equations with quadratic nonlinearities, namely the fractional KdV equation (fKdV)

\begin{align}\label{fKdV}
	\partial_t u-|D|^{\alpha} \partial_x u=-u\partial_x u, \quad -1<\alpha<0,
\end{align}
or the Whitham equation
\begin{equation}\label {Whit}
	\partial_t u+\mathcal L \partial_xu=-u\partial_x u.
\end{equation}

It was proven in \cite{EW} that the lifespan of solutions to \eqref{fKdV} with $-1<\alpha<1, \; \alpha \neq 0$ and with initial data of size $O(\epsilon)$ in $H^N(\R), N\geq 3$ is $O(1/\epsilon^2).$

This result is specially striking when $-1<\alpha<0$ since for $\alpha>0$ one expects global existence of small solutions. Note that for the inviscid Burgers equation this lifespan is $O(1/\epsilon).$ 

Proving a global existence result of small solutions of the  fKdV equation in the range $-1<\alpha<0$ or for the Whitham equation (and for the modified Whitham equation) is a challenging open question. 

On the other hand, a finite time blow-up by shock formation was observed via numerical approach and conjectured in \cite{MR3317254, KLPS}. This has been   now  proven in \cite{HT,Hur} both for the fKdV equation when $-1<\alpha<-\frac{1}{3}$ and for the Whitham equation.


\section{Appendix}
For the reader's  convenience, we list two technical lemmas proven in \cite{MR3552008} and \cite{MR3121725} respectively. 
The first one is the following interpolation inequality:
\begin{lemma}[\cite{MR3552008}]\label{interpolation} It holds that
	\begin{align}\label{75}
	\|\widehat{P_kg}\|_{L^\infty}^2\lesssim \left\|P_kg\right\|_{L^1}^2\lesssim 2^{-k}\|\widehat{P_kg}\|_{L^2}(\|\widehat{P_kg}\|_{L^2}+2^k\|\partial\widehat{P_kg}\|_{L^2}).
	\end{align}
\end{lemma}

\bigskip
The other one is the bound on pseudo-product operators satisfying certain strong integrability conditions:
\begin{lemma}[\cite{MR3121725}]\label{symbol estimate}
	Assume that \(L^1(\R\times\R)\) satisfies
	\begin{align*}
	\left\|\int_{\R^2}m(\eta,\sigma)e^{\mathrm{i}x\eta}e^{\mathrm{i}y\sigma}\,\diff \eta\diff \sigma\right\|_{L^1_{x,y}}\lesssim A,
	\end{align*} 
	for some \(A\in(0,\infty)\). Then for any \((p,q,r)\in\{(2,2,\infty),(2,\infty,2),(\infty,2,2)\}\),
	\begin{align*}
	\left|\int_{\R^2}m(\eta,\sigma)\widehat{f}(\eta)\widehat{g}(\sigma)\widehat{h}(-\eta-\sigma)\,\diff \eta\diff \sigma\right|\lesssim A\|f\|_{L^p}\|g\|_{L^q}\|h\|_{L^r}.
	\end{align*}
\end{lemma}

\vspace{0.5cm}
\noindent {\bf Acknowledgments.} The work of both authors was partially  supported by the ANR project ANuI ( ANR-17-CE40-0035-02).

\end{document}